\documentclass[preprint,3p,times]{elsarticle}
\usepackage[toc,page,title,titletoc,header]{appendix}
\usepackage{amsmath}
\usepackage{amssymb}
\usepackage{amsthm}

\newtheorem{example}{Example}

\usepackage{enumerate}
\usepackage{enumitem}
\usepackage{mathrsfs}
\usepackage{graphicx}
\usepackage{epsfig}
\usepackage{tikz,pgfplots,tikz-3dplot}
\usepackage{epstopdf}
\usepackage{microtype}
\usetikzlibrary{fit}
\biboptions{sort&compress}

\newtheorem{thm}{Theorem}[section]
\newtheorem{pro}{Proposition}[section]

\newtheorem{rem}{Remark}
\renewcommand{}

\newcommand{\nn}{\nonumber}

\newcommand{\ttau}{\Delta t}

\def\epsilon{\varepsilon} 
\newcommand{\mat}[1]{\boldsymbol{#1}}

\allowdisplaybreaks
\pgfplotsset{compat=1.18}

\begin{document}
\begin{frontmatter}
\title{
A sharp-interface approach for simulating solid-state dewetting of thin films with double-bubble structure}
\author[1]{Meng Li}
\author[1]{Nan Wang}
\address[1]{School of Mathematics and Statistics, Zhengzhou University,
Zhengzhou 450001, China.}
\ead{This author's research was supported by National Natural Science Foundation of China (Nos. 11801527,U23A2065,12001499), the China Postdoctoral Science Foundation (No. 2023T160589). 
limeng@zzu.edu.cn (Meng Li); 
Corresponding author: nwang@zzu.edu.cn}
\author[1]{Ruofan Zhao}
\author[1]{Chunjie Zhou}

\begin{abstract}
We develop a sharp-interface model for solid-state dewetting of double-bubble thin films using an energy variational approach based on a newly proposed interfacial energy. This model characterizes the dynamic evolution of interfaces in double-bubble thin films, a process primarily governed by surface diffusion and junction/contact points migration, and fundamentally distinct from the behavior observed in a single thin film. Subsequently, a structure-preserving parametric finite element approximation is developed for the sharp-interface model, which can preserve both area conservation and energy stability. Extensive numerical experiments are presented to demonstrate the convergence, structure-preserving properties, and superior mesh quality of the proposed method. Additionally, we investigate several specific evolution processes, including the equilibrium shapes of double-bubble thin films and the pinch-off dynamics of long islands.
\end{abstract}



\end{frontmatter}
\section{Introduction}\label{sec1}
Solid-state dewetting (SSD) is a well-documented phenomenon in solid–solid–vapor systems, often used to describe the agglomeration process of solid thin films on substrates. When a solid film is deposited onto a substrate, it is typically in an unstable or metastable state due to the combined effects of surface tension and capillarity. This instability can drive complex morphological transformations, such as fingering instabilities \cite{Kan05,Ye10b,Ye11a,Ye11b}, edge retraction \cite{Wong00,dornel2006surface,hyun2013quantitative}, faceting \cite{Jiran90,Jiran92,Ye10a}, and pinch-off events \cite{Jiang12,Kim15}. These processes are crucial for understanding the dynamics of thin-film evolution and can significantly impact the development of material structures in various technological applications.
SSD has become integral to a variety of cutting-edge technologies \cite{Mizsei93,Armelao06,Schmidt09}, sparking significant interest and extensive research into its underlying mechanisms. 
These efforts include both experimental work \cite{Jiran90,Jiran92,Ye10a,Ye10b,Ye11a,Amram12,Rabkin14,Herz216,Naffouti16,Naffouti17,Kovalenko17} and theoretical studies \cite{Wong00,Dornel06,hyun2013quantitative,Jiang12,Jiang16,Kim15,Kan05,Srolovitz86a,Srolovitz86,Wang15,baojcm2022,Bao17,Bao17b,Zucker16}. 
Although several models for SSD have been proposed under the assumption of isotropic surface energy \cite{Srolovitz86,Wong00,Dornel06,Jiang12,Zhao20}, the dynamics of SSD are strongly influenced by crystalline anisotropy, as evidenced by experimental results \cite{Thompson12,Leroy16}. 
A more thorough understanding of how crystalline anisotropy affects SSD is crucial, as it not only produces distinct phenomena but also plays a pivotal role in leveraging dewetting processes to fabricate intermediate structures for device development. 
In recent years, significant efforts have been devoted to developing numerical methods for investigating the influence of surface energy anisotropy on SSD, as discussed in \cite{Dornel06,Pierre09b,Dufay11,Klinger11shape,Zucker13,Bao17,Jiang16,Wang15,Jiang19a,Zhao19b} and related works.

A droplet or soap bubble naturally forms a spherical shape to minimize surface area for a given volume. The soap bubble cluster extends this problem to minimizing surface area for multiple enclosed regions with fixed volumes. 
These minimization problems have received considerable attention in the literature, as seen in both theoretical studies \cite{amilibia2001existence,foisy1993standard,wichiramala2004proof,paolini2020quadruple,hutchings2002proof} and numerical investigations \cite{sullivan1996open,taylor1976structure,amilibia2001existence,morgan2007colloquium,morgan1998wulff,wecht2000double}. 
A natural question that arises is how two or more thin films will evolve during SSD on a substrate. The existing research on SSD is only limited to the evolution of a single thin film, leaving the more complex interactions in multi-film systems unexplored. Investigating this problem would provide a more comprehensive understanding of the dewetting process in multilayer film systems, especially in applications at the micro- and nanoscale. 

To model this kind of anisotropic thin film solid materials, we first define the following surface energy density function:
\begin{align}
    \gamma( \mat  n):  \mathbb S ^1 \rightarrow  \mathbb  R^+,
\end{align}
where $\mat n$ denotes the unit outer normal vector of the crystalline interface, and $\mathbb S^1$ represents the unit circle. To model anisotropic solid materials effectively and concisely, we introduce the Cahn-Hoffman $\mat \xi$-vector formulation, defined as \cite{Hoffman72,Cahn74,wheeler1996xi,wheeler1999cahn,Eggleston01}
\begin{align}
    \mat\xi(\mat n)=\nabla\widehat{\gamma}(\mat n)\bigg|_{\mat p=\mat n}
    \qquad \text{with}\qquad 
    \widehat\gamma(\mat p)=|\mat p|\gamma\left(\frac{\mat p}{|\mat p|}\right),\quad \forall \mat p\in \mathbb R^2\backslash\{\mat 0\}.
\end{align}
Here $\widehat\gamma(\cdot)$ represents a homogeneous extension of the function $\gamma(\cdot)$, extending it from unit vectors to non-zero vectors. In two dimensions, the surface energy density $\gamma(\cdot)$ can also be expressed with \(\theta\) as the independent variable, such that \(\gamma(\theta)=\gamma(\mat n)\). In this case, the Cahn-Hoffman $\mat\xi$-vector holds 
\begin{align}
    \mat\xi=\gamma(\theta)\mat n-\gamma'(\theta)\mat \tau, 
\end{align}
where $\mat \tau=(\cos\theta, \sin\theta)^{\top}$ denotes the unit tangential vector and $\mat n = -\mat{\tau}^\perp$. We further define the surface stiffness as 
\begin{align}
    \widetilde{\gamma}(\theta)
    =\left[\mat H_\gamma(\mat n)\mat\tau\right]\cdot\mat\tau
    =\gamma(\theta)+\gamma''(\theta), 
\end{align}
where $\mat H_\gamma(\mat n)$ is the Hessian matrix, defined as 
$\mat H_\gamma(\mat n)=\nabla^2\widehat{\gamma}(\mat p)\big|_{\mat p=\mat n}$. Generally speaking, anisotropy can be classified into two cases: weakly anisotropic, where \(\widetilde{\gamma} > 0\) for all orientations, and strongly anisotropic, where there exist orientations for which \(\widetilde{\gamma} < 0\). It is important to note that the positivity of the surface stiffness is crucial for the stability of the Wulff shape \cite{winklmann2006note}, the well-posedness of the surface diffusion flow \cite{Jiang16}, and the numerical analysis of such problems \cite{Deckelnick05}.

In the kinetic evolution of SSD of thin films, the primary kinetic features are surface diffusion \cite{Mullins57} and contact line migration \cite{Wang15,Zhao17}. Srolovitz and Safran first proposed a sharp-interface model to study hole growth, assuming isotropic surface energy, small slope profile, and cylindrical symmetry \cite{Srolovitz86}. 
Building on the model, Wong et al. developed a “marker particle” method to solve the two-dimensional fully-nonlinear isotropic sharp-interface model without the small slope assumption, and further studied edge retraction of a semi-infinite step film \cite{Wong00}. Dornel et al. in \cite{dornel2006surface} developed a numerical scheme to study the pinch-off phenomenon in two-dimensional island films with high aspect ratios during SSD. Jiang et al. in \cite{Jiang12} proposed a phase-field model to simulate isotropic SSD of thin films, which effectively captures the topological changes during evolution. 
To study the effect of surface energy anisotropy, various approaches have been proposed \cite{Dornel06,Pierre09b,Dufay11,Carter95,Zucker13,Wang15,Jiang16,Bao17,garcke2023diffuse}. 
Based on the thermodynamic variation and the smooth vector-field perturbation method for the total interfacial energy, and utilizing the Cahn-Hoffman \(\mat \xi\)-vector, Jiang and Zhao \cite{Jiang19a} derived a sharp-interface model for weakly anisotropic surface energies.  Additionally, they proposed a parametric finite element method (PFEM) for numerically solving the sharp-interface model. 
However, this method lacks any structure-preserving properties. Following this, many researchers have continued to explore structure-preserving algorithms for anisotropic SSD models. Li and Bao \cite{li2021energy} considered an energy-stable PFEM for surface diffusion and SSD of thin films with weakly anisotropic surface energy. However, the numerical scheme does not conserve total discrete mass or area, and its theoretical analysis is not applicable to strongly anisotropic cases. Li et al. \cite{li2023symmetrized} proposed an accurate, area-conservative, and energy-stable PFEM for solving the sharp-interface continuum model of SSD with arbitrary anisotropic surface energy. 
Recent studies on structure-preserving algorithms have focused on the following areas, including the SSD on curved substrates \cite{bao2024structure1}, regularized SSD system \cite{li2024energy}, and axisymmetric SSD on curved-surface substrates \cite{duan2025solid}.

From all the work mentioned above, it is clear that the focus is solely on SSD involving a single thin film.
In practical materials science experiments, thin films undergoing SSD are often not isolated. The interaction between multiple films on the same surface is common, and such interactions can lead to more complex morphological evolution. This idea inspired our innovative study of SSD for thin films with a double-bubble structure. 
The idea can also be traced back to  several existing references by Barrett, Garcke and N{\"u}rnberg that focus on parametric finite element approximations  of interface
evolutions with triple junctions,
which is relevant for SSD. 
These references cover a range of topics, including isotropic surface diffusion of curve networks \cite{Barrett07}, isotropic/anisotropic surface diffusion and curvature flow of curve networks with boundary contacts \cite{Barrett07b,Barrett07Ani}, anisotropic surface diffusion and mean curvature flow of surface clusters with
boundary contacts \cite{Barrett10cluster} and non-neutral boundary contacts \cite{Barrett10}. 
In addition, Bao et al. in \cite{bao2023structure} studied the structure-preserving finite element approximation of surface diffusion for curve networks and surface clusters. 

As shown in Fig. \ref{fig:det}, there are three open curves \(\Gamma_{F_1V} = \mat{X}_1\left(x(s_1), y(s_1)\right)\), 
\(\Gamma_{{F_2V}} = \mat{X}_2\left(x(s_2), y(s_2)\right)\) and \(\Gamma_{{F_1F_2}} = \mat{X}_3\left(x(s_3), y(s_3)\right)\), 
corresponding to the left, right, and middle curves, respectively. Here, \(s_j\) represents the arc length of each of the three interfaces, respectively.  
The curves \(\Gamma_{F_1V}\) and \(\Gamma_{F_2V}\) represent the film/vapor interfaces that separate the two thin films from the vapor, while \(\Gamma_{F_1F_2}\) separates the two thin films. We denote the film/substrate interfaces as \(\Gamma_{F_1S_1}\) and \(\Gamma_{F_2S_2}\), and the vapor/substrate interfaces as \(\Gamma_{VS_1}\) and \(\Gamma_{VS_2}\).
The endpoints of the three curves lie on the substrate, labeled as A, B, and C, while the other endpoints of the three curves coincide at a single point, denoted as P. 
\begin{figure}[!ht]
\centering
\includegraphics[width=0.8\textwidth]{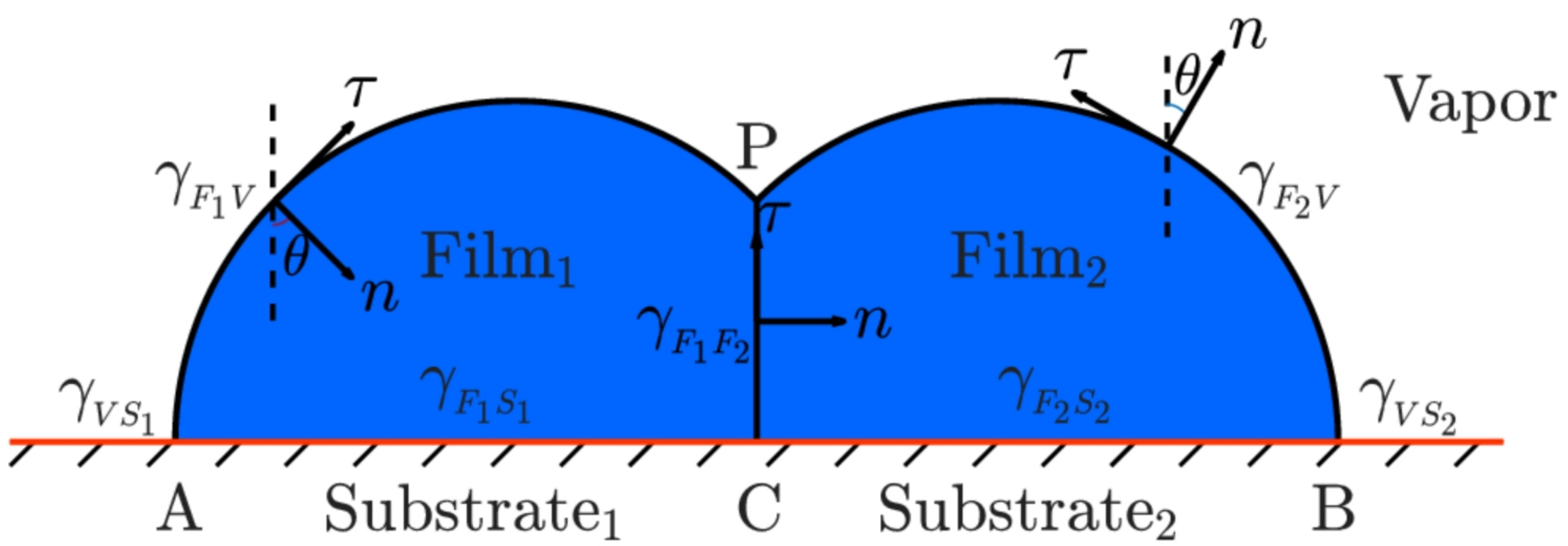}
\caption{A schematic description of the SSD. }
\label{fig:det}
\end{figure}
Then, the total interfacial energy of the double-bubble SSD system can be written as 
\begin{align}\label{eqn:ener0}
W(\Gamma)=&\int_{\Gamma_{F_1V}}\gamma_{_{F_1V}}d\Gamma_{F_1V}+\int_{\Gamma_{F_2V}}\gamma_{_{F_2V}}d\Gamma_{F_2V}
    +\int_{\Gamma_{F_1F_2}}\gamma_{_{F_1F_2}}d\Gamma_{F_1F_2}
    +\int_{\Gamma_{F_1S_1}}\gamma_{_{F_1S_1}}d\Gamma_{F_1S_1}+\int_{\Gamma_{F_2S_2}}\gamma_{_{F_2S_2}}d\Gamma_{F_2S_2}\nn\\
    &+\int_{\Gamma_{VS_1}}\gamma_{_{VS_1}}d\Gamma_{VS_1}+\int_{\Gamma_{VS_2}}\gamma_{_{VS_2}}d\Gamma_{VS_2},
\end{align}
where $\Gamma:=\left\{\Gamma_{F_1V}, \Gamma_{F_2V}, \Gamma_{F_1F_2}\right\}$, and $\gamma_{_{\mathcal U}}$ represents the surface energy density of the interface $\Gamma_{\mathcal U}$ (surface
energy per unit length). For SSD problems, we
often assume that $\gamma_{_{F_1S_1}}$, $\gamma_{_{F_2S_2}}$, $\gamma_{_{VS_1}}$ and 
$\gamma_{_{VS_2}}$ are constants. Meanwhile, 
the surface energy density functions $\gamma_{_{F_1V}}=\gamma_{_{F_1V}}(\mat n_{_{\Gamma_{_{F_1V}}}})$, $\gamma_{_{F_2V}}=\gamma_{_{F_2V}}(\mat n_{_{\Gamma_{_{F_2V}}}})$
and $\gamma_{_{F_1F_2}}=\gamma_{_{F_1F_2}}(\mat n_{_{\Gamma_{_{F_1F_2}}}})$ belong to the space $C^2(\mathbb S^1)$ and depend on the orientation of the film/vapor interface.
Under the assumptions, the total interfacial energy \eqref{eqn:ener0} can be simplified as 
\begin{align}\label{eqn:ener1x}
W(\Gamma)=&\int_{\Gamma_{F_1V}}\gamma_{_{F_1V}}(\mat n_{_{\Gamma_{_{F_1V}}}})d\Gamma_{F_1V}+\int_{\Gamma_{F_2V}}\gamma_{_{F_2V}}(\mat n_{_{\Gamma_{_{F_2V}}}})d\Gamma_{F_2V}
    +\int_{\Gamma_{F_1F_2}}\gamma_{_{F_1F_2}}(\mat n_{_{\Gamma_{_{F_1F_2}}}})d\Gamma_{F_1F_2}\nn\\
    &-\bigg(\gamma_{_{VS_1}}-\gamma_{_{F_1S_1}}\bigg)\bigg(x_C-x_A\bigg)
    -\bigg(\gamma_{_{VS_2}}-\gamma_{_{F_2S_2}}\bigg)\bigg(x_B-x_C\bigg).
\end{align}
In the above simplified energy, the first and second terms correspond to the film/vapor interfacial energy, the third term represents the film/film interfacial energy, and the last two terms account for the substrate energy.

The objectives of this paper are: (1) to define a total interfacial energy for the SSD of double-bubble thin films; (2) to derive the sharp-interface model by calculating the first variation of the energy using a smooth vector-field perturbation method; (3) to develop a structure-preserving PFEM (SP-PFEM) for the sharp-interface model with weakly/strongly anisotropic interfacial energy; (4) to check the convergence, structure-preservatity, and mesh quality of the SP-PFEM and investigate some physical insights into SSD through numerical simulations. 

The rest of the paper is organized as follows. 
In Section \ref{sec2}, based on the first variation of the total interfacial energy \eqref{eqn:ener1x}, computed using the smooth vector-field perturbation method, we derive the sharp-interface model for the SSD of double-bubble thin films.
In Section \ref{sec3}, by introducing a symmetric surface energy matrix, we reformulate the sharp-interface model into an equivalent symmetric and conservative system. Subsequently, we derive the variational formulation of the symmetric system and rigorously prove its area conservation and energy dissipation properties.
In Section \ref{sec4}, we establish a structure-preserving parametric finite element approximation for the variational formulation and demonstrate its area conservation and energy stability. Section \ref{sec5} presents extensive numerical experiments based on the proposed SP-PFEM. Finally, we summarize our findings and draw conclusions in Section \ref{sec6}.

\section{Sharp-interface model}\label{sec2}
In this section, we first introduce a smooth vector-field perturbation method to compute the first variation of the energy functional defined in \eqref{eqn:ener1x}. Subsequently, based on the first variation, we derive the sharp-interface model for the SSD of double-bubble thin films using the Cahn-Hoffman vector.
For simplicity, we introduce the notations $\Gamma_1$, $\Gamma_2$, and $\Gamma_3$ to represent $\Gamma_{F_1V}$, $\Gamma_{F_2V}$, and $\Gamma_{F_1F_2}$, respectively. The corresponding unit outer normal vectors are denoted by $\mat{n}_j$ for $j = 1, 2, 3$. Similarly, we use $\gamma_1$, $\gamma_2$ and $\gamma_3$ to denote $\gamma_{F_1V}$, $\gamma_{F_2V}$ and $\gamma_{F_1F_2}$, respectively. Then, we rewrite the total interfacial  
energy \eqref{eqn:ener1x} as 
\begin{align}\label{eqn:ener1}
W(\Gamma)=&\sum_{j=1}^3\int_{\Gamma_j}\gamma_j(\mat n_j)d\Gamma_j-\bigg(\gamma_{_{VS_1}}-\gamma_{_{F_1S_1}}\bigg)\bigg(x_C-x_A\bigg)
    -\bigg(\gamma_{_{VS_2}}-\gamma_{_{F_2S_2}}\bigg)\bigg(x_B-x_C\bigg).
\end{align}

\subsection{The first variation}

To begin, we introduce an independent parameter \(\rho_j \in \mathbb I_j = [0, 1]\) to parameterize a family of perturbed curves $$\Gamma_j^\epsilon=\mat X_j(\rho_j, \epsilon): \mathbb I_j\times [0, \epsilon_0]\rightarrow \mathbb R^2, \quad j=1,~2,~3;\qquad 
\Gamma^\epsilon=\bigg(\Gamma_1^\epsilon, \Gamma_2^\epsilon, \Gamma_3^\epsilon\bigg),$$ 
where \(\epsilon\in [0, \epsilon_0]\) controls the amplitude of the perturbation, with \(\epsilon_0\) representing the maximum perturbation amplitude. Obviously, it holds $\Gamma_j = \Gamma_j^0$, $j=1, 2, 3$. 
We further define the smooth perturbation vector field \(\mat{V}_j(\rho_j, \epsilon)\) of \(\Gamma_j\) as
\[
\mat{V}_j(\rho_j, \epsilon) = \frac{\partial \mat{X}_j(\rho_j, \epsilon)}{\partial \epsilon}, \quad j = 1, 2, 3;\qquad \mat V(\rho_j, \epsilon)=\bigg(\mat{V}_1(\rho_1, \epsilon), \mat{V}_2(\rho_2, \epsilon), \mat{V}_3(\rho_3, \epsilon)\bigg).
\]
Each point on the curve \(\Gamma_j\) is smoothly deformed according to the perturbation vector field \(\mat{V}_j\) as defined in the above equation. If the vector field \(\mat{V}_j\) is sufficiently smooth, the perturbed curves \(\Gamma_j\) will preserve the regularity of the original curve \(\Gamma_j\).
The following theorem present the first variation of the free energy functional \eqref{eqn:ener1}. 
\begin{thm}\label{thm:fv}
Under the junction point condition:
\begin{align}\label{eqn:cond1}
\sum_{j=1}^{3}\left[\gamma_j(\mat n_j)\mat\tau_j-\left(\mat{\xi}_{j}(\mat n_j)\cdot\mat\tau_j\right)\mat n_j\right]=\mat 0\qquad \text{at the junction point}~P,
\end{align}
    then the first variation of the free energy functional \eqref{eqn:ener1} for the double-bubble thin film SSD problem, 
 with respect to the smooth
deformation field $\mat V$, can be written as
\begin{align}\label{eqn:fv}
\delta W(\Gamma; \mat V) 
&= -\int_{\Gamma_1} \left (\partial_{s_1} \mat\xi_1^{\perp} \cdot \mat n_1 \right) \left( \mat V_{1, 0} \cdot \mat n_1 \right) ds_1 -\int_{\Gamma_2} \left( \partial_{s_2} \mat\xi_2^{\perp} \cdot \mat n_2 \right) \left( \mat V_{2, 0} \cdot \mat n_2 \right) ds_2 -\int_{\Gamma_3} \left(  \partial_{s_3} \mat\xi_3^{\perp} \cdot \mat n_3 \right) \left( \mat V_{3, 0} \cdot \mat n_3 \right) ds_3 \nn\\
&\quad-\Bigg[\bigg(\xi_{1, 2}-\left(\gamma_{VS_1} - \gamma_{F_1S_1}\right)\bigg)\left(\mat V_{1, 0}\cdot \mat e_1\right)\Bigg]\Bigg|_{s_1=0}
-\Bigg[\bigg(\xi_{2, 2}+\left(\gamma_{VS_2} - \gamma_{F_2S_2}\right)\bigg)\left(\mat V_{2, 0}\cdot \mat e_1\right)\Bigg]\Bigg|_{s_2=0} \nn\\
&
\quad-\Bigg[\bigg(\xi_{3, 2}+\left(\gamma_{VS_1} - \gamma_{F_1S_1}\right)-\left(\gamma_{VS_2} - \gamma_{F_2S_2}\right)\bigg)
\left(\mat V_{3, 0}\cdot \mat e_1\right)\Bigg]\Bigg|_{s_3=0},
\end{align}
where $\mat \xi_j=(\xi_{j, 1}, \xi_{j, 2})$, $\mat V_{j, 0}=\mat V_{j}(\rho_j,0)$, $j=1, 2, 3$. 
\end{thm}
\begin{proof}
Using the definition of the first variation with respect to the smooth perturbation vector-field, we have 
\begin{align}\label{eqn:first_pf1}
    &\delta W(\Gamma; \mat V)=\lim_{\epsilon \to 0} \frac{W(\Gamma^\epsilon) - W(\Gamma)}{\epsilon} \nonumber \\
    &=\lim_{\epsilon \to 0} \frac{\int_{\Gamma_{1}^{\epsilon}}\gamma_1(\mat n_1^\epsilon)ds_1^\epsilon  - \int_{\Gamma_1}\gamma_1(\mat n_1)ds_1 }{\epsilon}
+\lim_{\epsilon \to 0} \frac{\int_{\Gamma_{2}^{\epsilon}}\gamma_2(\mat n_2^\epsilon)ds_2^\epsilon  - \int_{\Gamma_2}\gamma_2(\mat n_2)ds_2 }{\epsilon}\nn\\
&~~~~+
\lim_{\epsilon \to 0} \frac{\int_{\Gamma_{3}^{\epsilon}}\gamma_3(\mat n_3^\epsilon)ds_3^\epsilon  - \int_{\Gamma_3}\gamma_3(\mat n_3)ds_3 }{\epsilon}
- \big(\gamma_{VS_1} - \gamma_{F_1S_1}\big)
\lim_{\epsilon \to 0}\frac{(x_C^\epsilon -x_C)-( x_A^\epsilon - x_A)}{\epsilon}\nn\\
    &\quad   - (\gamma_{VS_2} - \gamma_{F_2S_2})
    \lim_{\epsilon \to 0}\frac{(x_B^\epsilon - x_B) - (x_C^\epsilon  - x_C)}{\epsilon}\nn\\
    &=:\sum_{i=1}^{5}\delta W_i.
\end{align}
By using Taylor expansion, thanks to $\nabla\widehat{\gamma}_1(\mat n_1)\cdot \mat n_1=\gamma_1(\mat n_1)$, 
we obtain
\begin{align}\label{eqn:first_pf2}
\delta W_1 &=\lim_{\epsilon \to 0} \frac{\int_{\Gamma_{1}^{\epsilon}}\gamma_1(\mat n_1^\epsilon)ds_1^\epsilon  - \int_{\Gamma_1}\gamma_1(\mat n_1)ds_1 }{\epsilon}= \int_0^1 \left[\lim_{\epsilon \to 0} \frac{\gamma_1(\mat n_1^\epsilon) \left|\partial_{\rho_1} \mat X_1(\rho_1, \epsilon)\right| - \gamma_1(\mat n_1) \left|\partial_{\rho_1} \mat X_1(\rho_1, 0)\right|}{\epsilon}\right] d\rho_1\nn\\
&=\int_0^1 \lim_{\epsilon \to 0} \frac{1}{\epsilon} \Bigg\{\left[\gamma_1(\mat n_1)+
\nabla \widehat{\gamma}_1(\mat n_1) \cdot \left( \frac{\partial_{\rho_1} \mat X_1 \cdot \partial_{\rho_1} \mat V_{1, 0}}{\left|\partial_{\rho_1} \mat X_1\right|^3} \partial_{\rho_1} \mat X_1^\perp - \frac{\partial_{\rho_1} \mat V_{1, 0}^\perp}{\left|\partial_{\rho_1} \mat X_1\right|} \right)\epsilon\right]\left[\left|\partial_{\rho_1} \mat X_1(\rho_1, 0)\right|+\frac{\partial_{\rho_1} \mat X_1 \cdot \partial_{\rho_1} \mat V_{1, 0}}{\left|\partial_{\rho_1} \mat X_1\right|}\epsilon\right]\nn\\
&~~~~- \gamma_1(\mat n_1) \left|\partial_{\rho_1} \mat X_1(\rho_1, 0)\right|\Bigg\}d\rho_1\nn \\
&=\int_{0}^{1}\left[\gamma_1(\mat n_1)\frac{\partial_{\rho_1} \mat X_1 \cdot \partial_{\rho_1} \mat V_{1, 0}}{\left|\partial_{\rho_1} \mat X_1\right|}+\nabla \widehat{\gamma}_1(\mat n_1) \cdot \left( \frac{\partial_{\rho_1} \mat X_1 \cdot \partial_{\rho_1} \mat V_{1, 0}}{\left|\partial_{\rho_1} \mat X_1\right|^3} \partial_{\rho_1} \mat X_1^\perp - \frac{\partial_{\rho_1} \mat V_{1, 0}^\perp}{\left|\partial_{\rho_1} \mat X_1\right|} \right)\left|\partial_{\rho_1}\mat X_1(\rho_1,0)\right|   \right]d\rho_1\nn\\
&= - \int_0^1 \nabla \widehat{\gamma}_1(\mat n_1) \cdot \partial_{\rho_1} \mat V_{1, 0}^\perp \, d\rho_1 = \int_0^1 \nabla \widehat{\gamma}_1(\mat n_1)^\perp \cdot \partial_{\rho_1} \mat V_{1, 0} \, d\rho_1.
\end{align}
Taking $\nabla \widehat{\gamma}_1(\mat{n}_1) = \mat{\xi}_1$ in the above equation, by virtue of integration by parts, and due to the fact that $\partial_{s_1} \mat{\xi}_1 \parallel \mat{\tau}_1$ and $\partial_{s_1} \mat{\xi}_1^\perp \parallel \mat{n}_1$, we get
    \begin{align}\label{eqn:first_pf3}
\delta W_1 &= -\int_0^1 \partial_{\rho_1} \mat\xi_1^\perp \cdot \mat V_{1, 0} \, d\rho_1 + \left[\mat\xi_1^\perp \cdot \mat V_{1, 0}\right]\bigg|_{\rho_1=0}^{\rho_1=1} \nn\\
&= -\int_{\Gamma_1} \partial_{s_1} \mat\xi_1^\perp \cdot \mat V_{1, 0} \, ds_1 + \left[\mat\xi_1^\perp \cdot \mat V_{1, 0}\right]\bigg|_{s_1=0}^{s_1=L_1} \nn\\
&= -\int_{\Gamma_1} \left(\partial_{s_1} \mat\xi_1 ^{\perp} \cdot \mat n_1 \right)\left(\mat V_{1, 0} \cdot \mat n_1 \right) ds_1 + \left[ \mat\xi_1^{\perp} \cdot \mat V_{1, 0} \right]\bigg|_{s_1=0}^{s_1=L_1}. 
\end{align}
Similarly, we also have
\begin{align}
&\label{eqn:first_pf4} \delta W_2=-\int_{\Gamma_2} \left(\partial_{s_2} \mat\xi_2 ^{\perp} \cdot \mat n_2 \right)\left(\mat V_{2, 0} \cdot \mat n_2 \right) ds_2 + \left[ \mat\xi_2^{\perp} \cdot \mat V_{2, 0} \right]\bigg|_{s_2=0}^{s_2=L_2},\\
& \label{eqn:first_pf5} \delta W_3=-\int_{\Gamma_3} \left(\partial_{s_3} \mat\xi_3 ^{\perp} \cdot \mat n_3 \right)\left(\mat V_{3, 0} \cdot \mat n_3 \right) ds_3 + \left[ \mat\xi_3^{\perp} \cdot \mat V_{3, 0} \right]\bigg|_{s_3=0}^{s_3=L_3}.
\end{align}
In addition, we have 
\begin{align}\label{eqn:first_pf6}
\delta W_4 &=\left(\gamma_{VS_1} - \gamma_{F_1S_1}\right)\lim_{\epsilon\rightarrow 0}\frac{\left( x_A^\epsilon - x_A\right)-\left(x_C^\epsilon -x_C\right)}{\epsilon}=\left(\gamma_{VS_1} - \gamma_{F_1S_1}\right)
\bigg[\left(\mat V_{1, 0}\cdot \mat e_1\right)\bigg|_{s_1=0}-\left(\mathbf{V}_{3, 0}\cdot \mat e_1\right)\bigg|_{s_3=0}\bigg],
\end{align}
and 
\begin{align}\label{eqn:first_pf7}
\delta W_5 &=\left(\gamma_{VS_2} - \gamma_{F_2S_2}\right)\lim_{\epsilon\rightarrow 0}\frac{\left( x_C^\epsilon - x_C\right)-\left(x_B^\epsilon -x_B\right)}{\epsilon}=\left(\gamma_{VS_2} - \gamma_{F_2S_2}\right)
\bigg[\left(\mat V_{3, 0}\cdot \mat e_1\right)\bigg|_{s_3=0}-\left(\mathbf{V}_{2, 0}\cdot \mat e_1\right)\bigg|_{s_2=0}\bigg].
\end{align}
Thanks to the boundary condition \eqref{eqn:cond1} at the junction point, we have 
\begin{align}\label{eqn:first_pf8}
\sum_{j=1}^{3}\mat\xi_j^{\perp}\bigg|_{s_j=L_j}=\mat 0.
\end{align}
Since at the junction point $P$, we have
\begin{align*}
    \mat X_1(\rho_1, \epsilon)\bigg|_{\rho_1=1}=\mat X_2(\rho_2, \epsilon)\bigg|_{\rho_2=1}=\mat X_3(\rho_3, \epsilon)\bigg|_{\rho_3=1},
\end{align*}
which means that 
\begin{align}\label{eqn:first_pf8_add1}
    \mat V_{1, 0}\bigg|_{s_1=L_1}=\mat V_{2, 0}\bigg|_{s_2=L_2}=\mat V_{3, 0}\bigg|_{s_3=L_3}.
\end{align}
From \eqref{eqn:first_pf8} and \eqref{eqn:first_pf8_add1}, we have 
\begin{align}\label{eqn:first_pf9}
\left(\mat\xi_1^{\perp}\cdot\mat V_{1, 0}\right)\bigg|_{s_1=L_1}+\left(\mat\xi_2^{\perp}\cdot\mat V_{2, 0}\right)\bigg|_{s_2=L_2}
+\left(\mat\xi_3^{\perp}\cdot\mat V_{3, 0}\right)\bigg|_{s_3=L_3}
=0. 
\end{align}
Summing \eqref{eqn:first_pf3}-\eqref{eqn:first_pf7}, 
and by virtue of \eqref{eqn:first_pf9}, 
we obtain 
\begin{align}\label{eqn:first_pf10}
   \delta W(\Gamma; \mat V) &= -\int_{\Gamma_1} \left(\partial_{s_1} \mat\xi_1 ^{\perp} \cdot \mat n_1 \right)\left(\mat V_{1, 0} \cdot \mat n_1 \right) ds_1-\int_{\Gamma_2} \left(\partial_{s_2} \mat\xi_2 ^{\perp} \cdot \mat n_2 \right)\left(\mat V_{2, 0} \cdot \mat n_2 \right) ds_2-\int_{\Gamma_3} \left(\partial_{s_3} \mat\xi_3 ^{\perp} \cdot \mat n_3 \right)\left(\mat V_{3, 0} \cdot \mat n_3 \right) ds_3\nn\\
   &~~~~+\left[ \mat\xi_1^{\perp} \cdot \mat V_{1, 0} \right]\bigg|_{s_1=0}^{s_1=L_1}
   +\left[ \mat\xi_2^{\perp} \cdot \mat V_{2, 0} \right]\bigg|_{s_2=0}^{s_2=L_2} + \left[ \mat\xi_3^{\perp} \cdot \mat V_{3, 0} \right]\bigg|_{s_3=0}^{s_3=L_3}+\left(\gamma_{VS_1} - \gamma_{F_1S_1}\right)
\bigg[\left(\mat V_{1, 0}\cdot \mat e_1\right)\bigg|_{s_1=0}-\left(\mathbf{V}_{3, 0}\cdot \mat e_1\right)\bigg|_{s_3=0}\bigg]\nn\\
&~~~~+\left(\gamma_{VS_2} - \gamma_{F_2S_2}\right)
\bigg[\left(\mat V_{3, 0}\cdot \mat e_1\right)\bigg|_{s_3=0}-\left(\mathbf{V}_{2, 0}\cdot \mat e_1\right)\bigg|_{s_2=0}\bigg]\nn\\
&= -\int_{\Gamma_1} \left(\partial_{s_1} \mat\xi_1 ^{\perp} \cdot \mat n_1 \right)\left(\mat V_{1, 0} \cdot \mat n_1 \right) ds_1-\int_{\Gamma_2} \left(\partial_{s_2} \mat\xi_2 ^{\perp} \cdot \mat n_2 \right)\left(\mat V_{2, 0} \cdot \mat n_2 \right) ds_2-\int_{\Gamma_3} \left(\partial_{s_3} \mat\xi_3 ^{\perp} \cdot \mat n_3 \right)\left(\mat V_{3, 0} \cdot \mat n_3 \right) ds_3\nn\\
   &~~~~-\left[ \mat\xi_1^{\perp} \cdot \mat V_{1, 0} \right]\bigg|_{s_1=0}
   -\left[ \mat\xi_2^{\perp} \cdot \mat V_{2, 0} \right]\bigg|_{s_2=0} - \left[ \mat\xi_3^{\perp} \cdot \mat V_{3, 0} \right]\bigg|_{s_3=0}+\left(\gamma_{VS_1} - \gamma_{F_1S_1}\right)
\bigg[\left(\mat V_{1, 0}\cdot \mat e_1\right)\bigg|_{s_1=0}-\left(\mathbf{V}_{3, 0}\cdot \mat e_1\right)\bigg|_{s_3=0}\bigg]\nn\\
&~~~~+\left(\gamma_{VS_2} - \gamma_{F_2S_2}\right)
\bigg[\left(\mat V_{3, 0}\cdot \mat e_1\right)\bigg|_{s_3=0}-\left(\mathbf{V}_{2, 0}\cdot \mat e_1\right)\bigg|_{s_2=0}\bigg].
\end{align}
Since   
$\mat V_{j, 0}|_{s_j=0}\parallel \mat e_1$, 
we have 
\begin{align}\label{eqn:first_pf11}
    \left[ \mat\xi_j^{\perp} \cdot \mat V_{j, 0} \right]\bigg|_{s_j=0}
    =\bigg[ \left(\mat\xi_j^{\perp} \cdot \mat e_1\right)\left(\mat V_{j, 0}\cdot\mat e_1\right) \bigg]\bigg|_{s_j=0},\qquad j = 1, 2, 3. 
\end{align}
Using \eqref{eqn:first_pf11} in \eqref{eqn:first_pf10} obtains that 
\begin{align*}
\delta W(\Gamma; \mat V) 
&= -\int_{\Gamma_1} \left (\partial_{s_1} \mat\xi_1^{\perp} \cdot \mat n_1 \right) \left( \mat V_{1, 0} \cdot \mat n_1 \right) ds_1 -\int_{\Gamma_2} \left( \partial_{s_2} \mat\xi_2^{\perp} \cdot \mat n_2 \right) \left( \mat V_{2, 0} \cdot \mat n_2 \right) ds_2 -\int_{\Gamma_3} \left(  \partial_{s_3} \mat\xi_3^{\perp} \cdot \mat n_3 \right) \left( \mat V_{3, 0} \cdot \mat n_3 \right) ds_3 \nn\\
&\quad-\Bigg[\bigg(\xi_{1, 2}-\left(\gamma_{VS_1} - \gamma_{F_1S_1}\right)\bigg)\left(\mat V_{1, 0}\cdot \mat e_1\right)\Bigg]\Bigg|_{s_1=0}
-\Bigg[\bigg(\xi_{2, 2}+\left(\gamma_{VS_2} - \gamma_{F_2S_2}\right)\bigg)\left(\mat V_{2, 0}\cdot \mat e_1\right)\Bigg]\Bigg|_{s_2=0} \nn\\
&
\quad-\Bigg[\bigg(\xi_{3, 2}+\left(\gamma_{VS_1} - \gamma_{F_1S_1}\right)-\left(\gamma_{VS_2} - \gamma_{F_2S_2}\right)\bigg)
\left(\mat V_{3, 0}\cdot \mat e_1\right)\Bigg]\Bigg|_{s_3=0}.
\end{align*}
Therefore, we have completed the proof. 
\end{proof}

\subsection{Sharp-interface model via Cahn-Hoffman $\mat \xi_j$-vector formulation}
By applying Theorem \ref{thm:fv}, the first variation of the total energy functional \eqref{eqn:ener1} with respect to the interfaces $\Gamma_j$ ($j=1, 2, 3$) can be expressed as:  
\begin{align}\label{eqn:fv_inter}  
\frac{\delta W}{\delta \Gamma_j} = -\partial_{s_j} \mat\xi_j^{\perp} \cdot \mat{n}_j, \qquad j=1, 2, 3.  
\end{align}  
Furthermore, the first variation of the total energy functional \eqref{eqn:ener1} with respect to the three contact points $A$, $B$, and $C$ is given by:  
\begin{subequations}
\label{eqn:fv_con}
\begin{align}
& \frac{\delta W}{\delta x_A}
=-\xi_{1, 2}\bigg|_{s_1=0}+\left(\gamma_{VS_1} - \gamma_{F_1S_1}\right),\label{eqn:fv_con_a}\\
& \frac{\delta W}{\delta x_B}
=-\xi_{2, 2}\bigg|_{s_2=0}-\left(\gamma_{VS_2} - \gamma_{F_2S_2}\right), \label{eqn:fv_con_b}\\
& \frac{\delta W}{\delta x_C}
=-\xi_{3, 2}\bigg|_{s_3=0}-\left(\gamma_{VS_1} - \gamma_{F_1S_1}\right)+\left(\gamma_{VS_2} - \gamma_{F_2S_2}\right).\label{eqn:fv_con_c}
\end{align}
\end{subequations}

According to the Gibbs–Thomson relation 
\cite{Mullins57,Sutton95}, the chemical potentials of the system are defined as  
\begin{align}\label{eqn:cp}  
\mu_j = \Omega_{0, j} \frac{\delta W}{\delta \Gamma_j} = -\Omega_{0, j} \partial_{s_j} \mat{\xi}_j^{\perp} \cdot \mathbf{n}_j, \qquad j=1, 2, 3,  
\end{align}  
where \(\Omega_{0, j}\) denotes the atomic volume of the \(j\)-th thin film material.
The normal velocity of the \( j \)-th interface curve \( \Gamma_j \), denoted as \( \mat N_j \), is governed by the following surface diffusion flow \cite{Cahn74,Mullins57}:  
\begin{align}\label{eqn:JN}
\mat J_j = -\frac{D_{s_j} \nu_j}{k_{B, j} T_{e, j}} \nabla_{s_j} \mu_j, \qquad  
\mat N_j = -\Omega_{0, j} \left[ \nabla_{s_j} \cdot \mat J_j \right] = \frac{\Omega_{0, j} D_{s_j} \nu_j}{k_{B, j} T_{e, j}} \partial_{s_js_j} \mu_j,  
\end{align}  
where \( \mat J_j \) represents the surface flux, \( D_{s_j} \) is the surface diffusion coefficient, \( \nu_j \) is the atomic density, \( k_{B, j} \) is the Boltzmann constant, and \( T_{e, j} \) is the temperature. Here, \( k_{B, j} T_{e, j} \) represents the thermal energy of the \( j \)-th thin film in the system.
Furthermore, the motion of the three contact points is governed by the energy gradient flow, which is described by the time-dependent Ginzburg–Landau kinetic equations \cite{Wang15}, i.e.,
\begin{align}
    \frac{dx_A(t)}{dt}=-\eta_1\frac{\delta W}{\delta x_A},\quad \frac{dx_B(t)}{dt}=-\eta_2\frac{\delta W}{\delta x_B},\quad \frac{dx_C(t)}{dt}=-\eta_3\frac{\delta W}{\delta x_C}. 
\end{align}
Here, the parameters $\eta_j\in (0, +\infty)$, $j=1, 2, 3$ represent the finite mobility of the contact points. For a detailed physical interpretation of this approach, readers can refer to \cite{Wang15}. 

By selecting the characteristic length scale and characteristic surface energy scale of the \( j \)-th thin film as \( h_0 \) and \( \gamma_0 \), respectively, the time scale is defined as \( \frac{h_{0, j}^4}{Q_j \gamma_{0, j}} \), where \( Q_j = \frac{\Omega_{0, j}^2 D_{s_j} \nu_j}{k_{B, j} T_{e, j}} \). Additionally, the contact point mobility is scaled by \( \frac{Q_j}{h_{0, j}^3} \). 
Based on these scalings, we derive the following dimensionless sharp-interface model \cite{Wang15,Bao17,Jiang19a}:
\begin{subequations}\label{eqn:si}
\begin{align}
&\partial_{t} \mat{X}_j = \partial_{s_js_j} \mu_j \mat{n}_j, \quad 0 < s_j < L_j(t), \quad t > 0,\quad j = 1, 2, 3 \label{eqn:si1}\\
&\mu_j = -\partial_{s_j} \mat{\xi}_j^{\perp} \cdot \mat{n}_j, \quad \mat{\xi}_j = \nabla \widehat{\gamma}_j(\mat{p})|_{\mat{p}=\mat{n}_j}.\label{eqn:si2}
\end{align}
\end{subequations}
Here, \(\Gamma_j:=\Gamma_j(t)=\mat{X}_j(s_j, t)=(x_j(s_j, t), y_j(s_j, t))\) denotes the position of the \(j\)-th moving film/vapor interface, where \(s_j\) represents the arc length along the \(j\)-th interface, and \(\mat{n}_j\) is the outward unit normal vector to the \(j\)-th interface. The chemical potential is given by \(\mu_j(s, t)\), and \(\mat{\xi}_j\) represents the dimensionless Cahn-Hoffman vector. Additionally, \(L_j:=L_j(t)\) is the total length of the \(j\)-th interface. The initial conditions are given as 
\begin{align}\label{eqn:ini}
    \mat X_{j, 0}(s_j)
    =\left(x_{j, 0}(s_j), y_{j, 0}(s_j)\right),
    \qquad 
    0\leq s_j\leq L_{j, 0},\qquad j = 1, 2, 3, 
\end{align}
where $\mat X_{j, 0}(s_j)=\mat X_{j}(s_j, 0)$, $x_{j, 0}(s_j)=x_{j}(s_j, 0)$, $y_{j, 0}(s_j)=y_{j}(s_j, 0)$ and $L_{j, 0}=L_j(0)$. 
The initial conditions satisfy $0=y_{j, 0}(0)<y_{j, 0}(L_{j, 0})$. The boundary conditions are as follows:
\begin{itemize}
     \item [\textbf{(I)}] \textbf{Contact point condition}
 \begin{align}
\label{eqn:bd1}
 y_j(0, t)=0\qquad j=1, 2, 3, \qquad t \geq 0. 
 \end{align}
 \item [\textbf{(II)}] \textbf{Relaxed contact angle condition}
 \begin{align}
\label{eqn:bd2}
 \frac{dx_{A}}{dt}=\eta_1\bigg(\xi_{1, 2}\big|_{s_1=0}-\sigma_1\bigg), \quad 
 \frac{dx_{B}}{dt}=\eta_2\bigg(\xi_{2, 2}\big|_{s_2=0}+\sigma_2\bigg),\quad
 \frac{dx_{C}}{dt}= \eta_3\bigg(\xi_{3, 2}\big|_{s_3=0}+\sigma_1-\sigma_2\bigg),  \quad t \geq 0,
 \end{align}
 where $\sigma_1 = \frac{\gamma_{VS_1} - \gamma_{F_1S_1}}{\gamma_{0}}$ and $\sigma_2 = \frac{\gamma_{VS_2} - \gamma_{F_2S_2}}{\gamma_{0}}$ are dimensionless material constants, with  $\gamma_{0}$  representing the dimensionless unit of surface energy density.
 \item [\textbf{(III)}] \textbf{Zero-mass flux condition}
\begin{align}
\label{eqn:bd3}
\partial_{s_j}\mu_j(0, t)=0,
\quad j=1, 2, 3,\quad t\geq 0;\qquad 
\partial_{s_1}\mu_1(L_1, t)=\partial_{s_2}\mu_2(L_2, t)=\partial_{s_3}\mu_3(L_3, t),
\quad t\geq 0. 
\end{align}
\item [\textbf{(IV)}] \textbf{Junction point condition}
\begin{subequations}\label{eqn:bd4}
    \begin{align}
& \mat X_1(L_1, t)=\mat X_2(L_2, t)=\mat X_3(L_3, t),\qquad t\geq 0; \label{eqn:bd4_a}\\
&\sum_{j=1}^{3}\left[\gamma_j(\mat n_j)\mat\tau_j-\left(\nabla\gamma_{j}(\mat n_j)\cdot\mat\tau_j\right)\mat n_j\right]\bigg|_{s_j=L_j}=\mat 0,\qquad t\geq 0; \label{eqn:bd4_b}\\
&\sum_{j=1}^{3}\mu_j(L_j, t)=0,\qquad t\geq 0. \label{eqn:bd4_c}
\end{align}
\end{subequations}

\end{itemize}

\begin{rem}\label{rem1}
The boundary conditions described above serve the following purposes:  
\begin{itemize}
    \item [a)] Condition (I) ensures that the contact points are constrained to move along the substrate.
    \item [b)] Condition (II) facilitates the relaxation of the contact angle. 
    \item [c)] Condition (III) guarantees the conservation of the total area/mass, implying no mass flux at the contact points. 
    \item [d)] Condition (IV) specifies the requirements for the junction point \( P \). Among these conditions, \eqref{eqn:bd4_a} is naturally satisfied. As for \eqref{eqn:bd4_b}, it has already been utilized in the derivation of the first variation and is referred to as the force balance condition. Meanwhile, \eqref{eqn:bd4_c} can be interpreted as the chemical potential continuity condition.
\end{itemize}
\end{rem}

\subsection{Geometric properties}
In this subsection, we demonstrate that the sharp-interface model \eqref{eqn:si} satisfies the area conservation and energy dissipation laws. 
\begin{pro}\label{pro:laws} (Area Conservation $\&$ Energy Dissipation) 
Assume that \(\Gamma_j(t) = \mathbf{X}_j(s_j, t)\) for \(j = 1, 2, 3\) represents the solution to the sharp-interface model \eqref{eqn:si}, subject to the initial condition \eqref{eqn:ini} and the boundary conditions \eqref{eqn:bd1}–\eqref{eqn:bd4}. Then, 
the total area of the double-bubble thin film remains conserved throughout the evolution, i.e.,
\begin{align}\label{eqn:area_conser}
A(t) \equiv A(0),
\end{align}
where the total area $A(t)$ is defined by 
\begin{align}\label{eqn:area}
    A(t) = \int_{\Gamma_2(t)}y_2(s_2, t)\,\partial_{s_2}x_{2}(s_2, t)ds_2
    -\int_{\Gamma_1(t)}y_1(s_1, t)\,\partial_{s_1}x_{1}(s_1, t)ds_1,\qquad t\geq 0.
\end{align}
Additionally, the total free energy is dissipative in the sense that 
\begin{align}\label{eqn:ener_decay}
    E(t_1)\leq E(t_2)\leq E(0),\qquad 
    t_1\geq t_2\geq 0,
\end{align}
where the total free energy $E(t)$ is defined by 
\begin{align}\label{eqn:ener2}
    E(t) =\sum_{j=1}^3\int_{\Gamma_j(t)}\gamma_j(\mat n_j)d\Gamma_j-\sigma_1\bigg(x_C(t)-x_A(t)\bigg)
    -\sigma_2\bigg(x_B(t)-x_C(t)\bigg),\qquad t\geq 0.
\end{align}
\end{pro}
\begin{proof}
We use the Reynolds transport theorem to calculate the time derivative of the total area $A(t)$. Indeed, it is obvious that 
\begin{align}\label{eqn:laws_pf1}
   \frac{d}{dt}A(t)  & =\frac{d}{dt}\int_{\Gamma_2}y_2\,\partial_{s_2}x_{2}ds_2
    -\frac{d}{dt}\int_{\Gamma_1}y_1\,\partial_{s_1}x_{1}ds_1\nn\\
    &=\frac{d}{dt}\int_0^1y_2\partial_{\rho_2}x_2d\rho_2-\frac{d}{dt}\int_0^1y_1\partial_{\rho_1}x_2d\rho_1\nn\\
    &=\int_0^1\bigg[\partial_ty_2\partial_{\rho_2}x_2+y_2\partial_{\rho_2}\partial_tx_2\bigg]d\rho_2-\int_0^1\bigg[\partial_ty_1\partial_{\rho_1}x_1+y_1\partial_{\rho_1}\partial_tx_1\bigg]d\rho_1. 
\end{align}
By using integration by parts in \eqref{eqn:laws_pf1}, and thanks to the 
contact point condition \eqref{eqn:bd1} 
and the junction point condition \eqref{eqn:bd4_a}, we have 
\begin{align}\label{eqn:laws_pf2}
    \frac{d}{dt}A(t)  & = 
\int_0^1\bigg[\partial_ty_2\partial_{\rho_2}x_2-\partial_{\rho_2}y_2\partial_tx_2\bigg]d\rho_2
-\int_0^1\bigg[\partial_ty_1\partial_{\rho_1}x_1-\partial_{\rho_1}y_1\partial_tx_1\bigg]d\rho_1
+\left[y_2\partial_tx_2\bigg|_{\rho_2=0}^{\rho_2=1}-y_1\partial_tx_1\bigg|_{\rho_1=0}^{\rho_1=1}\right]\nn\\
& = 
\int_0^1\bigg[\partial_ty_2\partial_{\rho_2}x_2-\partial_{\rho_2}y_2\partial_tx_2\bigg]d\rho_2
-\int_0^1\bigg[\partial_ty_1\partial_{\rho_1}x_1-\partial_{\rho_1}y_1\partial_tx_1\bigg]d\rho_1\nn\\
&=\int_{\Gamma_2}\partial_t\mat X_2\cdot \mat n_2ds_2-\int_{\Gamma_1}\partial_t\mat X_1\cdot \mat n_1ds_1.
\end{align}
Then, from \eqref{eqn:si1} and \eqref{eqn:laws_pf2}, by using integration by parts, and thanks to the zero-mass flux condition \eqref{eqn:bd3}, we obtain 
\begin{align}\label{eqn:laws_pf3}
    \frac{d}{dt}A(t)&=\int_{\Gamma_2}\partial_{s_2s_2}\mu_2ds_2-\int_{\Gamma_1}\partial_{s_1s_1}\mu_1ds_1=\partial_{s_2}\mu_2\bigg|_{s_2=0}^{s_2=L_2}-\partial_{s_1}\mu_1\bigg|_{s_1=0}^{s_1=L_1}=0.
    \end{align}

For the derivative of the totoal free energy with respect to the time variable $t$, we can replace the perturbation variable $\epsilon$ in \eqref{eqn:fv} by the variable $t$. By this way, we have 
\begin{align}\label{eqn:laws_pf4}
    \frac{d}{dt}E(t)&=-\sum_{j=1}^3\int_{\Gamma_j} \left (\partial_{s_j} \mat\xi_j^{\perp} \cdot \mat n_j \right) \left( \partial_t\mat X_j \cdot \mat n_j\right) ds_j-\left(\xi_{1, 2}\bigg|_{s_1=0}-\sigma_1\right)\frac{dx_A}{dt}
-\left(\xi_{2, 2}\bigg|_{s_2=0}+\sigma_2\right)\frac{dx_B}{dt} 
\nn\\
&~~~~-\left(\xi_{3, 2}\bigg|_{s_3=0}+\sigma_1-\sigma_2\right)
\frac{dx_C}{dt},
\end{align}
where we have used the relations $\frac{dx_A}{dt}=\left(\mat V_{1, 0}\cdot \mat e_1\right)\big|_{s_1=0}$, $\frac{dx_B}{dt}=\left(\mat V_{2, 0}\cdot \mat e_1\right)\big|_{s_2=0}$ and $\frac{dx_C}{dt}=\left(\mat V_{3, 0}\cdot \mat e_1\right)\big|_{s_3=0}$. Substituting \eqref{eqn:si} into \eqref{eqn:laws_pf4}, there holds  
\begin{align}\label{eqn:laws_pf5}
    \frac{d}{dt}E(t)&=\sum_{j=1}^3\int_{\Gamma_j} \mu_j \partial_{s_js_j}\mu_j ds_j -\left(\xi_{1, 2}\bigg|_{s_1=0}-\sigma_1\right)\frac{dx_A}{dt}
-\left(\xi_{2, 2}\bigg|_{s_2=0}+\sigma_2\right)\frac{dx_B}{dt} -\left(\xi_{3, 2}\bigg|_{s_3=0}+\sigma_1-\sigma_2\right)
\frac{dx_C}{dt}. 
\end{align}
By virtue of the zero-mass flux condition \eqref{eqn:bd3} and the junction point condition \eqref{eqn:bd4_c}, we have 
\begin{align}\label{eqn:laws_pf6}
    \mu_1\partial_{s_1}\mu_1\bigg|_{s_1=0}^{s_1=L_1}+ \mu_2\partial_{s_2}\mu_2\bigg|_{s_2=0}^{s_2=L_2}+ \mu_3\partial_{s_3}\mu_3\bigg|_{s_3=0}^{s_3=L_3}=\left(\mu_1+\mu_2+\mu_3\right)\partial_{s_1}\mu_1\bigg|_{s_1=L_1}=0. 
\end{align}
Using integration by parts in \eqref{eqn:laws_pf5}, and with the help of \eqref{eqn:laws_pf6} and the relaxed contact angle condition \eqref{eqn:bd2}, we obtain 
\begin{align}\label{eqn:laws_pf7}
    \frac{d}{dt}E(t)=-\sum_{j=1}^3\int_{\Gamma_j}  \left(\partial_{s_j}\mu_j\right)^2 ds_j-\frac{1}{\eta_1}\left(\frac{dx_A}{dt}\right)^2
-\frac{1}{\eta_2}\left(\frac{dx_B}{dt}\right)^2 -\frac{1}{\eta_3}\left(
\frac{dx_C}{dt}\right)^2\leq 0.
\end{align}
Therefore, the energy dissipation \eqref{eqn:ener_decay} has been demonstrated. 
\end{proof}

\section{Variational formulation}\label{sec3}
To build the variational formulation of the sharp-interface model \eqref{eqn:si}, we first introduce a symmetric surface energy matrix: \cite{bao2023symmetrized,bao2024structure1,li2025structure}
\begin{equation}
\mat{Z}_{K, j}(\mat{n}_j) = \gamma_j(\mat{n}_j) \mat{I}_{2} - \mat{n}_j \mat \xi_j(\mat n_j)^{\top} - \mat \xi_j(\mat n_j) \mat{n}_j^{\top} + K_j(\mat{n}_j) \mat{n}_j \mat{n}_j^{\top}, \quad \mat{n}_j \in \mathbb S^{1},\quad j = 1, 2, 3, 
\end{equation}
where $\mat{I}_{2}$ is a $2\times 2$ identity matrix, and $K_j(\mat n_j)$ is a non-negative stabilizing function that can be prespecified. Then, the sharp-interface model \eqref{eqn:si} is equivalent to the following symmetric and conservative system: 
\begin{subequations}\label{eqn:model}
    \begin{align}
    & \partial_t\mat X_j\cdot\mat n_j-\partial_{s_js_j}\mu_j=0, \quad 0<s_j<L_j(t),\quad t> 0, \quad j = 1, 2, 3,\label{eqn:model_a}\\
    & \mu_j\mat n_j+\partial_{s_j}\bigg[\mat{Z}_{K, j}(\mat{n}_j)\partial_{s_j}\mat X_j\bigg]=0, \label{eqn:model_b}
\end{align}
with the initial conditions \eqref{eqn:ini} and the boundary conditions \eqref{eqn:bd1}-\eqref{eqn:bd4}.
\end{subequations}
Due to the equivalence of the models, the area conservation and energy dissipation properties of the system \eqref{eqn:model} can be naturally derived. 

We define the functional space 
\begin{align*}
    L^2(\mathbb I_j)=\bigg\{u: \mathbb I_j\rightarrow\mathbb R, \int_{\mathbb I_j}\left|u(\rho_j)\right|^2\partial_{\rho_j}s_jd\rho_j<+\infty\bigg\},\qquad j=1, 2, 3,
\end{align*}
equipped with the $L^2$ inner product 
\begin{align*}
    (u, v)_{\Gamma_j(t)}=\int_{\Gamma_j(t)}u(s_j)\,v(s_j)ds_j=\int_{\mathbb I_j}u\left(s_j(\rho_j, t)\right)\, v\left(s_j(\rho_j, t)\right)\,\partial_{\rho_j}s_jd\rho_j,\qquad j=1, 2, 3.
\end{align*}
This inner product can be directly extended to $\left[L^2(\mathbb I_j)\right]^2$, $j=1, 2, 3$. Furthermore, We define the Sobolev spaces: 
\begin{align*}
&H^1(\mathbb{I}_j)=\left \{ u:\mathbb{I}_j\to\mathbb{R},u\in L^2(\mathbb{I}_j)~~
\text{and}~~\partial_\rho u \in L^2(\mathbb{I}_j) \right \},\qquad 
\mat H^1(\mathbb{I}_j)=\left \{(u_1, u_2), u_1\in H^1(\mathbb{I}_j)~~\text{and}~~u_2\in H^1(\mathbb{I}_j) \right \},\\
&\mathbb V=\left\{\left(\mat v_1, \mat v_2, \mat v_3\right)\in \prod_{j=1}^{3} \mat H^{1} \left( \mathbb{I}_j \right),~ 
\mat v_j=(v_{j, 1}, v_{j, 2}),~
v_{j, 2}(0)=0,~ 
\mat v_1(1) =\mat v_2(1) = \mat v_3(1),~j = 1, 2, 3\right\},\\
&\mathbb W=\left\{ \left(\chi_1, \chi_2, \chi_3\right) \in \prod_{j=1}^{3} H^{1} \left( \mathbb{I}_j \right):~ \sum_{j=1}^{3} \chi_{j}(1) = 0\right\}. 
\end{align*}

We now proceed to derive the variational formulation of the system \eqref{eqn:model}. Multiplying each component of the test function \(\varphi = (\varphi_1, \varphi_2, \varphi_3) \in \mathbb W\) to \eqref{eqn:model_a}, integrating over \(\Gamma_j(t)\), applying integration by parts, and utilizing the zero-mass flux condition \eqref{eqn:bd3}, we derive
\begin{align}\label{eqn:vf_pf1}
    &\sum_{j=1}^3\bigg[\left(\partial_t\mat X_j\cdot\mat n_j, \varphi_j\right)_{\Gamma_j(t)}-\left(\partial_{s_js_j}\mu_j, \varphi_j\right)_{\Gamma_j(t)}\bigg]\nn\\
    &=\sum_{j=1}^3\left[\left(\partial_t\mat X_j\cdot\mat n_j, \varphi_j\right)_{\Gamma_j(t)}+\left(\partial_{s_j}\mu_j, \partial_{s_j}\varphi_j\right)_{\Gamma_j(t)}-\partial_{s_j}\mu_j\varphi_j\bigg|_{s_j=0}^{s_j=L_j}\right]\nn\\
    &=\sum_{j=1}^3\bigg[\left(\partial_t\mat X_j\cdot\mat n_j, \varphi_j\right)_{\Gamma_j(t)}+\left(\partial_{s_j}\mu_j, \partial_{s_j}\varphi_j\right)_{\Gamma_j(t)}\bigg]-\partial_{s_1}\mu_1\bigg|_{s_1=L_1}\bigg(\varphi_1(L_1)+\varphi_1(L_2)+\varphi_1(L_3)\bigg)\nn\\
    & = \sum_{j=1}^3\bigg[\left(\partial_t\mat X_j\cdot\mat n_j, \varphi_j\right)_{\Gamma_j(t)}+\left(\partial_{s_j}\mu_j, \partial_{s_j}\varphi_j\right)_{\Gamma_j(t)}\bigg]=0. 
\end{align}
Next, we multiply each component of the test function \(\mat\omega = (\mat\omega_1, \mat\omega_2, \mat\omega_3) \in \mathbb V\) to \eqref{eqn:model_b} and integrate over \(\Gamma_j(t)\). Then, by 
utilizing integration by parts, we obtain 
\begin{align}\label{eqn:vf_pf2}
&\sum_{j=1}^3\left[ \left(\mu_j\mat n_j, \mat\omega_j\right)_{\Gamma_j(t)}+\left(\partial_{s_j}\bigg[\mat{Z}_{K, j}(\mat{n}_j)\partial_{s_j}\mat X_j\bigg], \mat\omega_j\right)_{\Gamma_j(t)}\right] \nn\\
   &=\sum_{j=1}^3\bigg[\left(\mu_j\mat n_j, \mat\omega_j\right)_{\Gamma_j(t)}-\left(\mat{Z}_{K, j}(\mat{n}_j)\partial_{s_j}\mat X_j, \partial_{s_j}\mat\omega_j\right)_{\Gamma_j(t)}
   \bigg]+\sum_{j=1}^3\bigg[\left(\mat{Z}_{K, j}(\mat{n}_j)\partial_{s_j}\mat X_j\right)\cdot\mat\omega_j\bigg]\bigg|_{s_j=0}^{s_j=L_j}\nn\\
   &=\sum_{j=1}^3\bigg[\left(\mu_j\mat n_j, \mat\omega_j\right)_{\Gamma_j(t)}-\left(\mat{Z}_{K, j}(\mat{n}_j)\partial_{s_j}\mat X_j, \partial_{s_j}\mat\omega_j\right)_{\Gamma_j(t)}
   \bigg]+\sum_{j=1}^3\left(\mat\xi_j^\perp\cdot\mat\omega_j\right)\bigg|_{s_j=0}^{s_j=L_j}=0,
\end{align}
where we have used the relation:
\begin{align*}
    \mat{Z}_{K, j}(\mat{n}_j)\partial_{s_j} \mathbf{X}_j= \gamma_j(\mat n_j) \mat \tau_j - (\mat \xi_j(\mat n_j) \cdot \mat \tau_j)\mat n_j=\mat\xi_j^{\perp},\qquad j = 1, 2, 3. 
\end{align*}
For the last term on the left-hand side of \eqref{eqn:vf_pf2}, since \(\mat\omega = (\mat\omega_1, \mat\omega_2, \mat\omega_3) \in \mathbb V\), and by using the junction point condition \eqref{eqn:bd4_b} and the relaxed contact angle condition, we have 
\begin{align}\label{eqn:vf_pf3}
&\sum_{j=1}^3\left(\mat\xi_j^\perp\cdot\mat\omega_j\right)\bigg|_{s_j=0}^{s_j=L_j}
=\sum_{j=1}^3\left(\mat\xi_j^\perp\cdot\mat\omega_j\right)\bigg|_{s_j=L_j}-\sum_{j=1}^3\left(\mat\xi_j^\perp\cdot\mat\omega_j\right)\bigg|_{s_j=0}\nn\\
&=\mat\omega_1(L_1)\cdot\sum_{j=1}^3\mat\xi_j^\perp\bigg|_{s_j=L_j}-\bigg[\xi_{1, 2}(0)\omega_{1, 1}(0)+\xi_{2, 2}(0)\omega_{2, 1}(0)+\xi_{3, 2}(0)\omega_{3, 1}(0)\bigg]\nn\\
&=-\bigg[\xi_{1, 2}(0)\omega_{1, 1}(0)+\xi_{2, 2}(0)\omega_{2, 1}(0)+\xi_{3, 2}(0)\omega_{3, 1}(0)\bigg]\nn\\
&=-\bigg[\frac{1}{\eta_1}\frac{dx_A}{dt}+\sigma_1\bigg]\omega_{1, 1}(0)
-\bigg[\frac{1}{\eta_2}\frac{dx_B}{dt}-\sigma_2\bigg]\omega_{2, 1}(0)
-\bigg[\frac{1}{\eta_3}\frac{dx_C}{dt}+\sigma_2-\sigma_1\bigg]\omega_{3, 1}(0). 
\end{align}
From \eqref{eqn:vf_pf1}, \eqref{eqn:vf_pf2} and \eqref{eqn:vf_pf3}, we derive a symmetrized variational formulation for the system \eqref{eqn:model} with the initial conditions \eqref{eqn:ini} and the boundary conditions \eqref{eqn:bd1}-\eqref{eqn:bd4}: given $\Gamma_j(0)=\mat X_j(\rho_j, 0)=\mat X_{j, 0}(s_j), j = 1, 2, 3$ with $s_j=\rho_jL_j$, find the evolution triple junction curves $\Gamma(t)=\mat X(\rho, t)=\left(\mat X_1(\rho_1, t), \mat X_2(\rho_2, t), \mat X_3(\rho_3, t)\right)\in \mathbb V$, and the chemical potential 
$\mu(\rho, t)=\left(\mu_1(\rho_1, t), \mu_2(\rho_2, t), \mu_3(\rho_3, t)\right)\in \mathbb W$, such that 
\begin{subequations}
 \label{eqn:vf}
\begin{align}
    & \sum_{j=1}^3\bigg[\left(\partial_t\mat X_j\cdot\mat n_j, \varphi_j\right)_{\Gamma_j(t)}+\left(\partial_{s_j}\mu_j, \partial_{s_j}\varphi_j\right)_{\Gamma_j(t)}\bigg]=0,\qquad \varphi=\left(\varphi_1, \varphi_2, \varphi_3\right)\in \mathbb W,\label{eqn:vf_a}\\
    & \sum_{j=1}^3\bigg[\left(\mu_j\mat n_j, \mat\omega_j\right)_{\Gamma_j(t)}-\left(\mat{Z}_{K, j}(\mat{n}_j)\partial_{s_j}\mat X_j, \partial_{s_j}\mat\omega_j\right)_{\Gamma_j(t)}
   \bigg]-\bigg[\frac{1}{\eta_1}\frac{dx_A}{dt}+\sigma_1\bigg]\omega_{1, 1}(0)
-\bigg[\frac{1}{\eta_2}\frac{dx_B}{dt}-\sigma_2\bigg]\omega_{2, 1}(0)\nn\\
&-\bigg[\frac{1}{\eta_3}\frac{dx_C}{dt}+\sigma_2-\sigma_1\bigg]\omega_{3, 1}(0)=0,\qquad \mat\omega = \left(\mat\omega_1, \mat\omega_2, \mat\omega_3\right) \in \mathbb V,
   \label{eqn:vf_b}
\end{align}
\end{subequations}
where $x_A(t)=x_1(\rho_1=0, t)$, $x_B(t)=x_2(\rho_2=0, t)$ and $x_C(t)=x_3(\rho_3=0, t)$. 
\begin{pro}\label{pro:laws1} (Area Conservation $\&$ Energy Dissipation for Variational Formulation) 
Let \(\left(\mat X(\rho, t), \mu(\rho, t)\right) \in \mathbb V \times \mathbb W\) be a solution to the variational formulation \eqref{eqn:vf}. Then, the total area \(A(t)\) of the double-bubble thin film is conserved, as expressed in \eqref{eqn:area_conser}, and the total energy \(E(t)\) is dissipative, as shown in \eqref{eqn:ener_decay}.
\end{pro}
\begin{proof}
From \eqref{eqn:laws_pf2}, taking \(\varphi = (1, 0, 0)\), \(\varphi = (0, 1, 0)\), and \(\varphi = (0, 0, 1)\) respectively, we obtain
\[ \left(\partial_t \mat X_j \cdot \mat n_j, \varphi_j \right)_{\Gamma_j(t)} = 0, \quad j = 1, 2, 3, \] 
which, in conjunction with \eqref{eqn:laws_pf2}, leads to area conservation. 

Next, we intend to prove the energy dissipation of the variational formulation \eqref{eqn:vf}. By using $\partial_t\mat n_j=-\left(\mat n_j\cdot\partial_{s_j}\partial_t\mat X_j\right)\mat \tau_j$, we have 
\begin{align}\label{eqn:vf_ener_decay_pf1}
  & \frac{d}{dt}\int_{\Gamma_j(t)}\gamma_j(\mat n_j)d\Gamma_j 
   =\frac{d}{dt}\int_0^1\gamma_j(\mat n_j)\partial_{\rho_j}s_jd\rho_j 
   =\int_0^1\frac{d}{dt}\gamma_j(\mat n_j)\partial_{\rho_j}s_jd\rho_j 
   +\int_0^1\gamma_j(\mat n_j)\frac{d}{dt}\left|\partial_{\rho_j}\mat X_j\right|d\rho_j \nn\\
   &~~~=\int_0^1\nabla\gamma_j(\mat n_j)\frac{d\mat n_j}{dt}\partial_{\rho_j}s_jd\rho_j 
   +\int_0^1\gamma_j(\mat n_j)\frac{\partial_{\rho_j}\mat X_j\cdot\partial_t\partial_{\rho_j}\mat X_j}{\left|\partial_{\rho_j}\mat X_j\right|}d\rho_j \nn\\
   &~~~=-\int_{\Gamma_j(t)}\left[\nabla\gamma_j(\mat n_j)\cdot\mat\tau_j\right]
   \left[\mat n_j\cdot\partial_{s_j}\partial_t\mat X_j\right]ds_j 
   +\int_{\Gamma_j(t)}\gamma_j(\mat n_j)
   \left[\mat \tau_j\cdot\partial_{s_j}\partial_t\mat X_j\right]
  ds_j \nn\\
  &~~~=\left(\mat Z_{K, j}(\mat n_j)\partial_{s_j}\mat X_j, \partial_{s_j}\partial_t\mat X_j\right)_{\Gamma_j(t)}.
\end{align}
From \eqref{eqn:vf_ener_decay_pf1}, we have 
\begin{align}\label{eqn:vf_ener_decay_pf2}
    \frac{d}{dt}E(t) =&
    \sum_{j=1}^3\left[\frac{d}{dt}\int_{\Gamma_j(t)}\gamma_j(\mat n_j)d\Gamma_j\right]-\sigma_1\bigg(\frac{dx_C(t)}{dt}-\frac{dx_A(t)}{dt}\bigg)
    -\sigma_2\bigg(\frac{dx_B(t)}{dt}-\frac{dx_C(t)}{dt}\bigg)\nn\\
    =&\sum_{j=1}^3\left(\mat Z_{K, j}(\mat n_j)\partial_{s_j}\mat X_j, \partial_{s_j}\partial_t\mat X_j\right)_{\Gamma_j(t)}-\sigma_1\bigg(\frac{dx_C(t)}{dt}-\frac{dx_A(t)}{dt}\bigg)-\sigma_2\bigg(\frac{dx_B(t)}{dt}-\frac{dx_C(t)}{dt}\bigg). 
\end{align}
Denoting  $\varphi=\left(\mu_1, 0, 0\right)$ in \eqref{eqn:model_a} and $\mat\omega=\left(\partial_t\mat X_1, \mat 0, \mat 0\right)$ in \eqref{eqn:model_b}, 
taking the sum of the resulting equations, we have 
\begin{align}\label{eqn:vf_ener_decay_pf3}
    \left(\mat Z_{K, 1}(\mat n_1)\partial_{s_1}\mat X_1, \partial_{s_1}\partial_t\mat X_1\right)_{\Gamma_1(t)}+\sigma_1\frac{dx_A}{dt}
    =-\frac{1}{\eta_1}\left(\frac{dx_A}{dt}\right)^2-\left(\partial_{s_1}\mu_1, \partial_{s_1}\mu_1\right)_{\Gamma_1(t)}. 
\end{align}
Similarly, by taking \(\varphi = (0, \mu_2, 0)\) in \eqref{eqn:model_a} and \(\mat \omega = (\mat 0, \partial_t \mat X_2, \mat 0)\) in \eqref{eqn:model_b}, or alternatively, by taking \(\varphi = (0, 0, \mu_3)\) in \eqref{eqn:model_a} and \(\mat \omega = (\mat 0, \mat 0, \partial_t \mat X_3)\) in \eqref{eqn:model_b}, we obtain 
\begin{align}
& \left(\mat Z_{K, 2}(\mat n_2)\partial_{s_2}\mat X_2, \partial_{s_2}\partial_t\mat X_2\right)_{\Gamma_2(t)}-\sigma_2\frac{dx_B}{dt}
    =-\frac{1}{\eta_2}\left(\frac{dx_B}{dt}\right)^2-\left(\partial_{s_2}\mu_2, \partial_{s_2}\mu_2\right)_{\Gamma_2(t)},\label{eqn:vf_ener_decay_pf4}\\
& \left(\mat Z_{K, 3}(\mat n_3)\partial_{s_3}\mat X_3, \partial_{s_3}\partial_t\mat X_3\right)_{\Gamma_3(t)}-\left(\sigma_1-\sigma_2\right)\frac{dx_C}{dt}
    =-\frac{1}{\eta_3}\left(\frac{dx_C}{dt}\right)^2-\left(\partial_{s_3}\mu_3, \partial_{s_3}\mu_3\right)_{\Gamma_3(t)}. \label{eqn:vf_ener_decay_pf5}
\end{align}
By combining \eqref{eqn:vf_ener_decay_pf3}, \eqref{eqn:vf_ener_decay_pf4} and \eqref{eqn:vf_ener_decay_pf5}, and from \eqref{eqn:vf_ener_decay_pf2}, we derive 
\begin{align}
    \frac{d}{dt}E(t) = 
    -\sum_{j=1}^3\left(\partial_{s_j}\mu_j, \partial_{s_j}\mu_j\right)_{\Gamma_j(t)}-\frac{1}{\eta_1}\left(\frac{dx_A}{dt}\right)^2
    -\frac{1}{\eta_2}\left(\frac{dx_B}{dt}\right)^2
    -\frac{1}{\eta_3}\left(\frac{dx_C}{dt}\right)^2\leq 0,
\end{align}
which implies the total energy dissipation. 
\end{proof}

\section{Parametric finite element approximation}\label{sec4}

In this section, we propose a structure-preserving parametric finite element approximation for the variational formulation \eqref{eqn:vf}. The time interval \([0, T]\) is discretized into \(M\) subintervals, such that \([0, T] = \cup_{k=0}^{M-1}[t_k, t_{k+1}]\), where the time step sizes are given by \(\ttau_k = t_{k+1} - t_k\) for \(k \geq 0\). 
Additionally, each spatial domain \(\mathbb{I}_j\) for \(j = 1, 2, 3\) is divided into \(N_j\) subintervals, such that \(\mathbb{I}_j = \cup_{k=1}^{N_j} \mathbb{I}_{j,k} = \cup_{k=1}^{N_j} [q_{j, k-1}, q_{j, k}]\), with nodes defined as \(q_{j, k} = k h_j\), for \(k = 0, 1, \dots, N_j\), and the spatial step size in each domain is \(h_j = N_j^{-1}\).

Define the finite element spaces 
\begin{align*}
\mathbb K ^h (\mathbb I_j) = \left\{u \in C(\mathbb I_j):~ u\big|_{\mathbb I_{j, k}}\in\mathbb P_1,~ k=1, 2\dots ,N_j\right\}\subseteq H^1(\mathbb I_j),
\qquad j = 1, 2, 3, 
\end{align*}
which $\mathbb P_1$ denotes the space of polynomials with degree at most $1$. 
We further define the spaces: 
\begin{align*}
    \mathbb V^h=\left[\prod_{j=1}^{3}\left(\mathbb K ^h (\mathbb I_j)\right)^2\right]\bigcap \mathbb V,\qquad \mathbb W^h=\left[\prod_{j=1}^{3}\mathbb K ^h (\mathbb I_j)\right]\bigcap \mathbb W.
\end{align*}

We denote 
\begin{align*}
    \Gamma^m:=\bigg(\Gamma_1^m, \Gamma_2^m, \Gamma_3^m\bigg)
    =\mat X^m:=\bigg(\mat X_1^m, \mat X_2^m, \mat X_3^m\bigg)\in \mathbb V^h,\qquad 
    \mu^m=\bigg(\mu_1^m, \mu_2^m, \mu_3^m\bigg)\in \mathbb W^h
\end{align*}
as the approximations of the triple junction curves $\Gamma(t_m):=\mat X(\cdot, t_m)\in \mathbb V$ and the chemical potential $\mu(\cdot, t_m)\in \mathbb W$, respectively. 
For simplicity, we
denote $\mat X_j^m=(x_j^m, y_j^m)$, $j=1, 2, 3$. 
Each approximation solution $\Gamma_j ^m$, $j=1, 2, 3$ is made up of the line segments
\begin{equation*}
\mat h_{j, k}^m := \mat X_j^m(\rho_k) - \mat X_j^m(\rho_{k-1}),\quad j=1, 2, 3, \quad k= 1, 2, \dots , N_j, 
\end{equation*}
with $| \mat h_{j, k}^m| $ denoting the length of $\mat h_{j, k}^m$. 
The unit tangential vector $\mat\tau_{j, k}^m$ and the outward unit normal vector $\mat n_{j, k}^m$ on each interval $\mathbb I_j$ can be computed by 
\begin{equation}
\mat\tau_j ^m\bigg|_{\mathbb I_k} = \frac{\mat h_{j, k}^m}{\left | \mat h_{j, k}^m\right |} := \mat\tau_{j, k}^m, \qquad \mat n_j^m\bigg|_{\mathbb I_k} = -\frac{\left(\mat h_{j, k}^m\right)^\bot}{\left |\mat h_{j, k}^m \right |} := \mat n_{j, k}^m ,\quad j=1, 2, 3, \quad k= 1, 2, \dots , N_j.
\end{equation}
Furthermore, we can define the mass lumped inner product $(\cdot, \cdot)_{\Gamma_j ^m}^h$ for two functions $\mat u$ and $\mat v$ with possible jumps at the nodes $\left\{\rho_{j, k}\right\}_{k=0}^{N_j}$ as follows. 
\begin{equation}
(\mat u, \mat v)_{\Gamma_j ^m}^h := \frac{1}{2}\sum_{k=1}^{N_j}{\left| \mat h_{j, k}^m \right|}\left[( \mat u\cdot \mat v)(\rho_{j, k}^-)+( \mat u\cdot \mat v)(\rho_{j, k-1}^+)\right], \qquad j=1, 2, 3, 
\end{equation}
where $\mat v(\rho_{j, k}^\pm ) = \lim_{\rho \to \rho_{j, k}^{\pm}} \mat v(\rho)$ for $0\le k\le N_j$. 

In what follows, given the triple junction curves \(\Gamma^m\), we aim to design a numerical method to compute \(\mat X^{m+1} \in \mathbb{V}^h\), which subsequently determines the updated triple junction curves \(\Gamma^{m+1}\). More importantly, the numerical method should preserve key geometric properties similar to those of the continuous model, including area conservation and energy stability. Motivated by \cite{jiang2021perimeter,bao2021structure,li2025structure}, we introduce a time-weighted approximation: 
\begin{align}\label{eqn:n}
    \mat n_j^{m+\frac12}=-\frac{1}{2}
    \bigg[\partial_{s_j}\mat X_j^m+\partial_{s_j}\mat X_j^{m+1} \bigg]^\perp
    =-\frac{\bigg[\partial_{\rho_j}\mat X_j^m+\partial_{\rho_j}\mat X_j^{m+1} \bigg]^\perp}{2\left|\partial_{\rho_j}\mat X_j^m\right|},\qquad j = 1, 2, 3.
\end{align}
Then, we build the following symmetrized SP-PFEM for the variational formulation \eqref{eqn:vf}. Given $\mat X^0=\left(\mat X_1^0, \mat X_2^0, \mat X_3^0\right)\in\mathbb V^h$, find $\left(\mat X^{m+1}, \mu^{m+1}\right)\in \mathbb V^h\times \mathbb W^h$ with $\mat X^{m+1}=\left(\mat X_1^{m+1}, \mat X_2^{m+1}, \mat X_3^{m+1}\right)$ and  $\mu^{m+1}=\left(\mu_1^{m+1}, \mu_2^{m+1}, \mu_3^{m+1}\right)$, such that 
\begin{subequations}\label{eqn:num}
    \begin{align}
        & \sum_{j=1}^3\left[\left(\frac{\mat X_j^{m+1}-\mat X_j^m}{\Delta t_m}\cdot\mat n_j^{m+\frac12}, \varphi_j^h\right)^h_{\Gamma_j^m}
        +\left(\partial_{s_j}\mu_j^{m+1}, \partial_{s_j}\varphi_j^h\right)^h_{\Gamma_j^m}
        \right]=0, \qquad \varphi^h=\left(\varphi_1^h, \varphi_2^h, \varphi_3^h\right) \in \mathbb W^h, \label{eqn:num_a}\\
        & \sum_{j=1}^3\left[\left(\mu_j^{m+1}\mat n_j^{m+\frac12}, \mat\omega_j^h\right)^h_{\Gamma_j^m}-\left(\mat{Z}_{K, j}(\mat{n}_j^m)\partial_{s_j}\mat X_j^{m+1}, \partial_{s_j}\mat\omega_j^h\right)^h_{\Gamma_j^m}
   \right]-\left[\frac{1}{\eta_1}\frac{x_A^{m+1}-x_A^m}{\Delta t_m}+\sigma_1\right]\omega_{1, 1}^h(0)
\nn\\
&-\left[\frac{1}{\eta_2}\frac{x_B^{m+1}-x_B^m}{\Delta t_m}-\sigma_2\right]\omega_{2, 1}^h(0)-\left[\frac{1}{\eta_3}\frac{x_C^{m+1}-x_C^m}{\Delta t_m}+\sigma_2-\sigma_1\right]\omega_{3, 1}^h(0)=0,\qquad \mat\omega^h = \left(\mat\omega_1^h, \mat\omega_2^h, \mat\omega_3^h\right) \in \mathbb V^h.
   \label{eqn:num_b}
    \end{align}
\end{subequations}

\begin{rem}
The numerical method described in \eqref{eqn:num} is implicit and can be efficiently solved using iterative techniques such as Picard’s or Newton’s method. When \(\mat n_j^{m+\frac12}\) is replaced by \(\mat n_j^m\), the resulting scheme becomes linear. Although this modification results in a loss of area conservation, the scheme still preserves energy stability. We refer to this method as ES-PFEM, and will provide comparisons in the numerical experiments.
However, the absence of area conservation may affect the accuracy of geometric features, especially in long-term simulations.
\end{rem}

Define the total area of the enclosed
area by the triple-junction piecewise linear curves and the substrates as 
\begin{align}\label{eqn:num_area}
    A^m=\frac12\sum_{k=1}^{N_2}\left(x_{2, k}^m-x_{2, k-1}^m\right)\left(y_{2, k}^m+y_{2, k-1}^m\right)-\frac12\sum_{k=1}^{N_1}\left(x_{1, k}^m-x_{1, k-1}^m\right)\left(y_{1, k}^m+y_{1, k-1}^m\right),\qquad m\geq 0. 
\end{align}
Then, for the numerical scheme \eqref{eqn:num}, we have the following area conservation property. 
\begin{pro}\label{pro:num_area_conser}
Let $\left(\mat X^m, \mu^m\right)\in \mathbb V^h\times \mathbb W^h$ be a solution of the numerical scheme \eqref{eqn:num}. Then, the total area defined in \eqref{eqn:num_area} is conservative in the sense that 
\begin{align}\label{eqn:num_area_conser}
    A^m\equiv A^0,\qquad m\geq 0. 
\end{align}
\end{pro}
\begin{proof}
We define a approximate solution $\Gamma_j^h(\alpha)=\mat X_j^h(\rho_j, \alpha)$ by the linear interpolation of $\mat X_j^{m}$ and $\mat X_j^{m+1}$, satisfying that 
\begin{align}\label{eqn:num_area_conser_pf1}
    \mat X_j^h(\rho_j, \alpha)=(1-\alpha)\mat X_j^m(\rho_j)+\alpha\mat X_j^{m+1}(\rho_j),\qquad 0\leq \alpha\leq 1. 
\end{align}
We denote by $\mat n_j^h$ the outward unit normal vector of $\Gamma_j^h(\alpha)$. Denote $\Gamma^h(\alpha)=\left(\Gamma_1^h(\alpha), \Gamma_2^h(\alpha), \Gamma_3^h(\alpha)\right)$, and $\mat X_j^h=\left(x_j^h, y_j^h\right)$ for $j=1, 2, 3$. 
Define $A^h(\alpha)$ as the total area enclosed by the triple junction curves $\Gamma^h(\alpha)$ and the substrates, given by 
\begin{align}\label{eqn:num_area_conser_pf2}
    A^h(\alpha)=\int_{\Gamma_2^h(\alpha)}y_2^h(s_2, \alpha)\partial_{s_2}x_2^h(s_2, \alpha)ds_2-\int_{\Gamma_1^h(\alpha)}y_1^h(s_1, \alpha)\partial_{s_1}x_1^h(s_1, \alpha)ds_1.
\end{align}
Using the Reynolds transport theorem, we obtain 
\begin{align}\label{eqn:num_area_conser_pf3}
  \frac{dA^h(\alpha)}{d\alpha} =
  \int_{\Gamma_2^h(\alpha)}\partial_\alpha\mat X_2^h\cdot\mat n_2^hds_2
  -
  \int_{\Gamma_1^h(\alpha)}\partial_\alpha\mat X_1^h\cdot\mat n_1^hds_1. 
\end{align}
Substituting \eqref{eqn:num_area_conser_pf1} into \eqref{eqn:num_area_conser_pf3} gives 
\begin{align}\label{eqn:num_area_conser_pf4}
  \frac{dA^h(\alpha)}{d\alpha} =
  &-
  \int_{\mathbb I_2}\bigg[\mat X_2^{m+1}-\mat X_2^{m}\bigg]\cdot\bigg[(1-\alpha)\partial_{\rho_2}\mat X_2^m+\alpha\partial_{\rho_2}\mat X_2^{m+1} \bigg]^\perp d\rho_2\nn\\
&+  \int_{\mathbb I_1}\bigg[\mat X_1^{m+1}-\mat X_1^{m}\bigg]\cdot\bigg[(1-\alpha)\partial_{\rho_1}\mat X_1^m+\alpha\partial_{\rho_1}\mat X_1^{m+1} \bigg]^\perp d\rho_1.
\end{align}
By integrating \eqref{eqn:num_area_conser_pf4} with respect to the variable $\alpha$ from $0$ to $1$, we get 
\begin{align}\label{eqn:num_area_conser_pf5}
 A^{m+1}-A^m=A^h(1)-A^h(0)=\bigg(\left[\mat X_2^{m+1}-\mat X_2^{m}\right]\cdot\mat n_2^{m+\frac12}, 1\bigg)_{\Gamma_2^m}^h
   -\bigg(\left[\mat X_1^{m+1}-\mat X_1^{m}\right]\cdot\mat n_1^{m+\frac12}, 1\bigg)_{\Gamma_1^m}^h.
\end{align}
From \eqref{eqn:num_area_conser_pf5}, and taking $\varphi^h=\left(\Delta t_m, 0, 0\right)$ and $\varphi^h=\left(0, \Delta t_m, 0\right)$ in \eqref{eqn:num_a} respectively, we finally arrive at $A^{m+1}-A^m=0$, which further implies the area conservation \eqref{eqn:num_area_conser}.  
\end{proof}

Define the total discrete free energy 
\begin{align}\label{eqn:num_ener}
    E^m=\sum_{j=1}^3\sum_{k=1}^{N_j}\left|\mat h_{j, k}^m\right|\gamma_j(\mat n_{j, k}^m)-\sigma_1\left(x_C^m-x_A^m\right)-\sigma_2\left(x_B^m-x_C^m\right),\qquad m\geq 0. 
\end{align}
We then obtain the following energy stability of the numerical scheme \eqref{eqn:num}. 
\begin{pro}
 Denote $\left(\mat X^m, \mu^m\right)\in \mathbb V^h\times \mathbb W^h$ as a solution of the numerical scheme \eqref{eqn:num}. Then, the total free energy defined in \eqref{eqn:num_ener} is stable in the sense that 
\begin{align}\label{eqn:num_ener_stable}
    E^{m+1}\leq E^m\leq \cdots\leq E^0, \qquad m\geq 0, 
\end{align}
provided that 
\begin{align}\label{eqn:num_ener_cond}
    \gamma(-\mat n_j)< 3\gamma(\mat n_j),\qquad \forall \mat n_j\in\mathbb S^1,
\end{align}
and the stability function $k(\mat n_j)$ is taken sufficiently large. 
\end{pro}
\begin{proof}   
Substituting $\varphi^h=\left(\Delta t_m\mu_1^{m+1}, 0, 0\right)$ in \eqref{eqn:num_a} and $\omega^h=\left(\mat X_1^{m+1}-\mat X_1^m, 0, 0\right)$ in \eqref{eqn:num_b}, then taking the sum of two resulting equations, we can obtain that 
\begin{align}\label{eqn:num_ener_stable_pf1}
    \left(\mat{Z}_{K, 1}(\mat{n}_1^m)\partial_{s_1}\mat X_1^{m+1}, \partial_{s_1}\mat X_1^{m+1}-\partial_{s_1}\mat X_1^{m}\right)_{\Gamma_1^m}^h
    +\sigma_1\left(x_A^{m+1}-x_A^m\right)
    =-\frac{\left(x_A^{m+1}-x_A^m\right)^2}{\eta_1\Delta t_m}
    -\Delta t_m \left(\partial_{s_1}\mu_1^{m+1}, \partial_{s_1}\mu_1^{m+1}\right)^h_{\Gamma_1^m}.
\end{align}
Similarly, taking $\varphi^h=\left(0, \Delta t_m\mu_2^{m+1}, 0\right)$ in \eqref{eqn:num_a} and $\omega^h=\left(0, \mat X_2^{m+1}-\mat X_2^m, 0\right)$ in \eqref{eqn:num_b}, or alternatively, taking $\varphi^h=\left(0, 0, \Delta t_m\mu_3^{m+1}\right)$ in \eqref{eqn:num_a} and $\omega^h=\left(0, 0, \mat X_3^{m+1}-\mat X_3^m\right)$ in \eqref{eqn:num_b}, we obtain 
\begin{align}
&\left(\mat{Z}_{K, 2}(\mat{n}_2^m)\partial_{s_2}\mat X_2^{m+1}, \partial_{s_2}\mat X_2^{m+1}-\partial_{s_2}\mat X_2^{m}\right)_{\Gamma_2^m}^h
    -\sigma_2\left(x_B^{m+1}-x_B^m\right)
    =-\frac{\left(x_B^{m+1}-x_B^m\right)^2}{\eta_2\Delta t_m}
    -\Delta t_m \left(\partial_{s_2}\mu_2^{m+1}, \partial_{s_2}\mu_2^{m+1}\right)^h_{\Gamma_2^m},\label{eqn:num_ener_stable_pf2}\\
    & \left(\mat{Z}_{K, 3}(\mat{n}_3^m)\partial_{s_3}\mat X_3^{m+1}, \partial_{s_3}\mat X_3^{m+1}-\partial_{s_3}\mat X_3^{m}\right)_{\Gamma_3^m}^h
    +\left(\sigma_2-\sigma_1\right)\left(x_C^{m+1}-x_C^m\right)
    =-\frac{\left(x_C^{m+1}-x_C^m\right)^2}{\eta_3\Delta t_m}
    -\Delta t_m \left(\partial_{s_3}\mu_3^{m+1}, \partial_{s_3}\mu_3^{m+1}\right)^h_{\Gamma_3^m}.
    \label{eqn:num_ener_stable_pf3}
\end{align}
Combining \eqref{eqn:num_ener_stable_pf1}, \eqref{eqn:num_ener_stable_pf2} and \eqref{eqn:num_ener_stable_pf3}, we have 
\begin{align}\label{eqn:num_ener_stable_pf4}
   & \sum_{j=1}^3 \left(\mat{Z}_{K, j}(\mat{n}_j^m)\partial_{s_j}\mat X_j^{m+1}, \partial_{s_j}\mat X_j^{m+1}-\partial_{s_j}\mat X_j^{m}\right)_{\Gamma_j^m}^h
    +\sigma_1\left(x_A^{m+1}-x_A^m\right)-\sigma_2\left(x_B^{m+1}-x_B^m\right)+\left(\sigma_2-\sigma_1\right)\left(x_C^{m+1}-x_C^m\right)\nn\\
    &=-\frac{\left(x_A^{m+1}-x_A^m\right)^2}{\eta_1\Delta t_m}
    -\frac{\left(x_B^{m+1}-x_B^m\right)^2}{\eta_2\Delta t_m}
    -\frac{\left(x_C^{m+1}-x_C^m\right)^2}{\eta_3\Delta t_m}
    -\Delta t_m\sum_{j=1}^3 \left(\partial_{s_j}\mu_j^{m+1}, \partial_{s_j}\mu_j^{m+1}\right)^h_{\Gamma_j^m}.
\end{align}
From \cite[Lemma 3.2]{bao2024structure1}, we know that for sufficiently large $k(\mat n_j)$, there holds 
    \begin{align}\label{eqn:num_ener_stable_pf5}
     \left(\mat{Z}_{K, j}(\mat{n}_j^m)\partial_{s_j}\mat X_j^{m+1}, \partial_{s_j}\mat X_j^{m+1}-\partial_{s_j}\mat X_j^{m}\right)_{\Gamma_j^m}^h
     \geq \left(\gamma(\mat n_j^{m+1}), 1\right)^h_{\Gamma_j^{m+1}}-\left(\gamma(\mat n_j^{m}), 1\right)^h_{\Gamma_j^{m}}. 
    \end{align}
Using \eqref{eqn:num_ener_stable_pf5} in \eqref{eqn:num_ener_stable_pf4}, we can conclude the energy stability of the numerical scheme \eqref{eqn:num}. 
\end{proof}

\section{Numerical results}\label{sec5}
In this section, we present several numerical examples using the SP-PFEM method to investigate the SSD behavior of double-bubble thin films. These simulations demonstrate the structure-preserving properties of the proposed method, including energy stability and area conservation, as well as its convergence, mesh quality, and computational efficiency.
The mesh ratio $R^h(t)$ and the loss area $\triangle A(t)$ at $t_m$ are defined as follows:
\begin{equation}
R^h(t)|_{t = t_m} :=
\text{max}_{1\le j\le 3}\frac{\text{max}_{1\le k\le N} \left | \mat X_{j,k}^m - \mat X_{j,k - 1}^m \right | }{\text{min}_{1\le k\le N} \left | \mat X_{j,k}^m - \mat X_{j, k - 1}^m \right | }, \quad \triangle A(t)|_{t = t_m} := A^m - A^0,
\quad
{E}(t)|_{t = t_m} = E^m,
\quad m\ge 0. 
\end{equation}
In the numerical experiments, we select two distinct anisotropic energy functions:
\begin{itemize}
    \item 2-fold anisotropy: $\gamma_j(\theta)=1+\beta_j\cos(2\theta),  \quad j = 1, 2, 3$;
    \item 4-fold anisotropy: $\gamma_j(\theta)=1+\beta_j\cos(4\theta),\quad j = 1, 2, 3$,
\end{itemize}
where $\beta_j$ denotes the degree of anisotropy of $j$-th curve. For the $k$-fold anisotropy, when $\beta_j = 0$, it represents isotropic; when $0 < \beta_j <\frac{1}{k^2-1}$, it represents weakly anisotropic; when $\beta_j \geq \frac{1}{k^2-1}$, it represents strongly anisotropic. 
During all the tests, we select the tolerance tol$=10e-8$.

\begin{example}\label{exa1}
(Convergence tests)
We test convergence by quantifying the difference between the two triple curves 
using the manifold distance {\cite{Zhao20}, defined by} 
\begin{equation*}
\text{Md}(\Gamma, \widehat{\Gamma}):=\left |(\Omega_1\backslash\Omega_2)\cup(\Omega_2\backslash\Omega_1) \right |=\left |\Omega_1 \right |+\left |\Omega_2 \right |-2\left |\Omega_1\cap\Omega_2 \right |, \nonumber
\end{equation*}
where $\Omega_1$ denotes the region enclosed by the triple curve $\Gamma=\{\Gamma_1, \Gamma_2, \Gamma_3\}$ and the substrate, while $\Omega_2$ denotes the region enclosed by the other triple curve $\widehat{\Gamma}=\{\widehat\Gamma_1, \widehat\Gamma_2, \widehat\Gamma_3\}$ and corresponding substrate. Here, $| \cdot |$ represents the area of the region.
Let $\mat X_j^m$ denote numerical approximation of surface with mesh size $h$ and time step $\ttau$, then introduce approximate solution between interval $[t_m, t_{m+1}]$ as
\begin{equation}
\mat X_{j, \ttau}^{h}(\rho_j, t)=\frac{t_{m+1}-t}{\ttau}\mat X_j^m(\rho_j)+\frac{t-t_m}{\ttau}\mat X_j^{m+1}(\rho_j), \quad\rho_j\in\mathbb{I}_j,\quad j=1, 2, 3. 
\end{equation}
This gives the evolving curve $\Gamma_{h, \ttau} = \mat X_{h, \ttau}(\mathbb{I}, t)$, then we further define the errors by  
\begin{equation*}
e_{h, \ttau}(t)=\text{Md}(\Gamma_{h, \ttau}, \Gamma_{\frac{h}{2}, \frac{\ttau}{4}}). 
\end{equation*}
In this example, we choose the initial data as
\begin{equation*}
    \mat X_1^0(\rho_1) = \left(4 + 2\cos\left(\pi-\frac{\pi\rho_1}{2}\right), \sin\left(\pi-\frac{\pi\rho_1}{2}\right)\right),  \quad \mat X_2^0(\rho_2) = \left(4 + 2\cos\left(\frac{\pi\rho_2}{2}\right), \sin\left(\frac{\pi\rho_2}{2}\right)\right),\quad  \mat X_3^0(\rho_3)=(4,\rho_3).
\end{equation*}
The numerical errors and orders of the SP-PFEM with 2-fold and 4-fold anisotropies are shown in Figures \ref{fig:2}-\ref{fig:5}, respectively. We can observe that the order of $O({\ttau} + h^2)$ is consistently exhibited across different anisotropic cases.
\end{example}

\begin{example}\label{exa2} (Mesh quality tests)
In this example, we use the same initial data as Example \ref{exa1}, and we investigate the impact of different parameters on the mesh ratio \( R^h(t) \) for both types of anisotropic energy functions. As observed in Figures \ref{fig:6}-\ref{fig:7}, the mesh ratio \( R^h(t) \) exhibits minor fluctuations but ultimately stabilizes over time. This behavior demonstrates that the SP-PFEM is capable of maintaining good mesh quality throughout long-term evolution.   
\end{example}

\begin{figure}[!htp]
\centering
\includegraphics[width=0.49\textwidth]{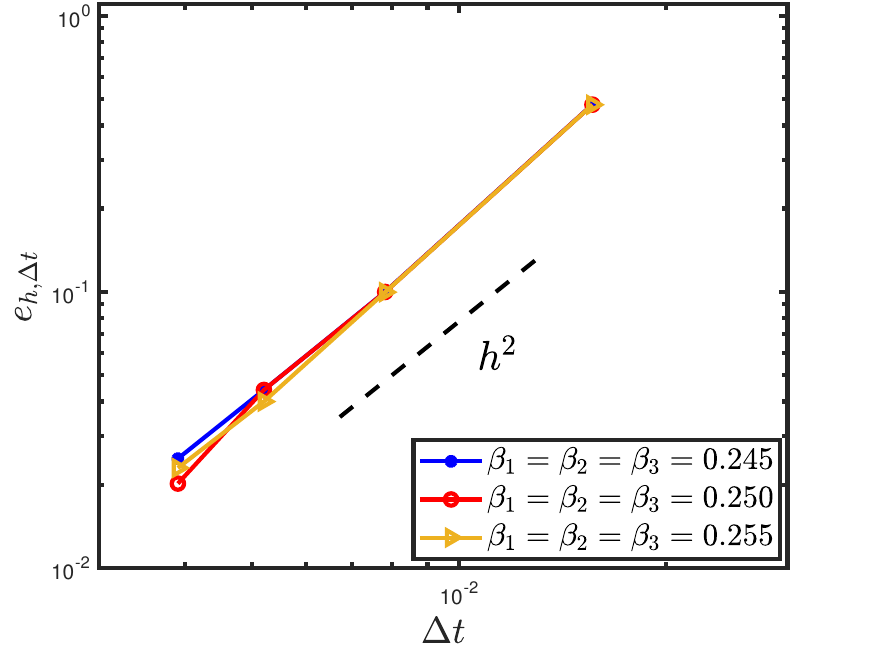}
\includegraphics[width=0.49\textwidth]{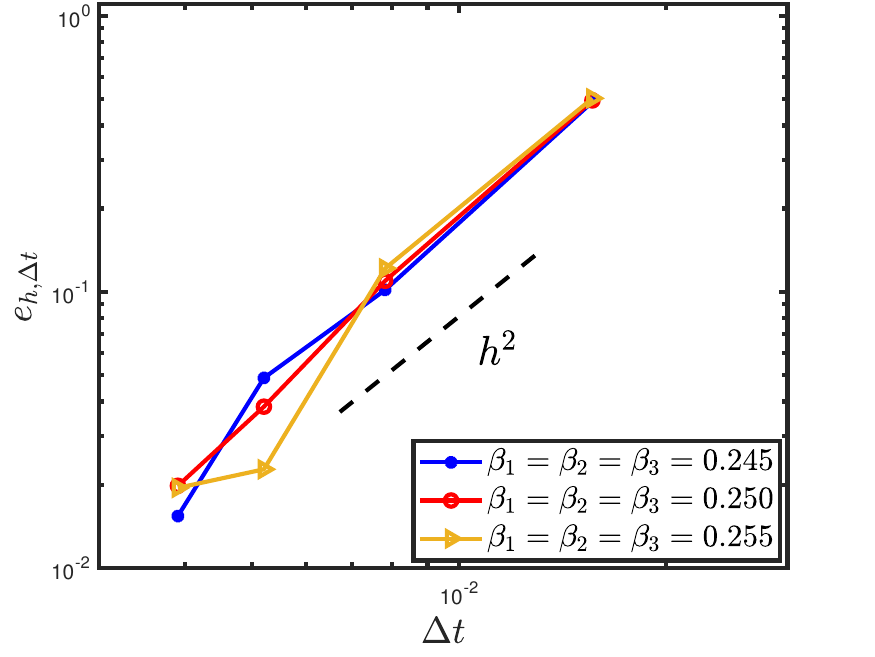}
\caption{Plot of numberical errors employing SP-PFEM at $\sigma_1 =\sigma_2 = -0.7$ (left panel) and $\sigma_1 =\sigma_2 = 0.7$ (right panel) for $2$-fold anisotropy. The parameters are selected as $\eta_1 =\eta_2 =\eta_3 = 100$, and $t_m = 1$. } 
\label{fig:2}
\end{figure}

\begin{figure}[!htp]
\centering
\includegraphics[width=0.49\textwidth]{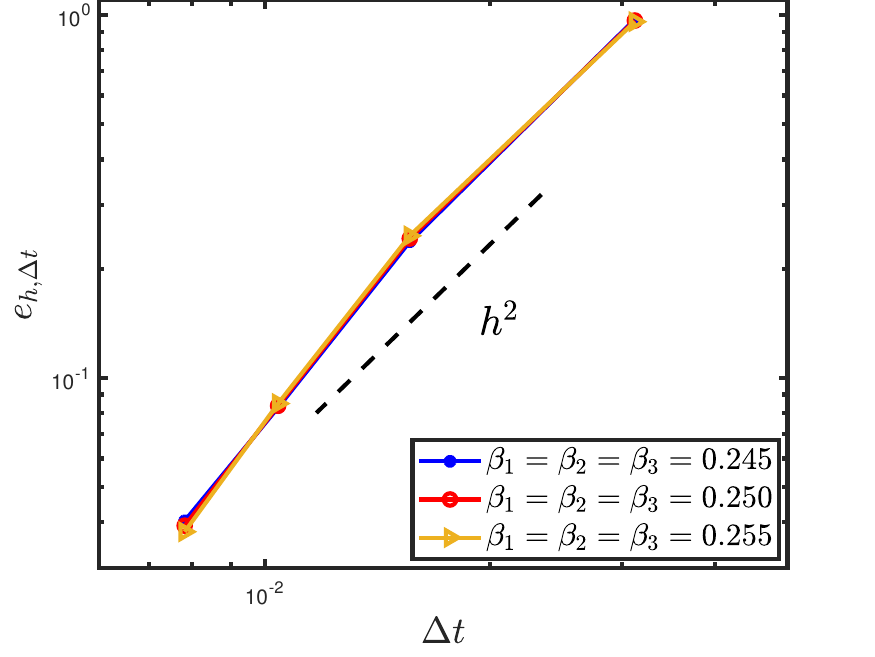}
\includegraphics[width=0.49\textwidth]{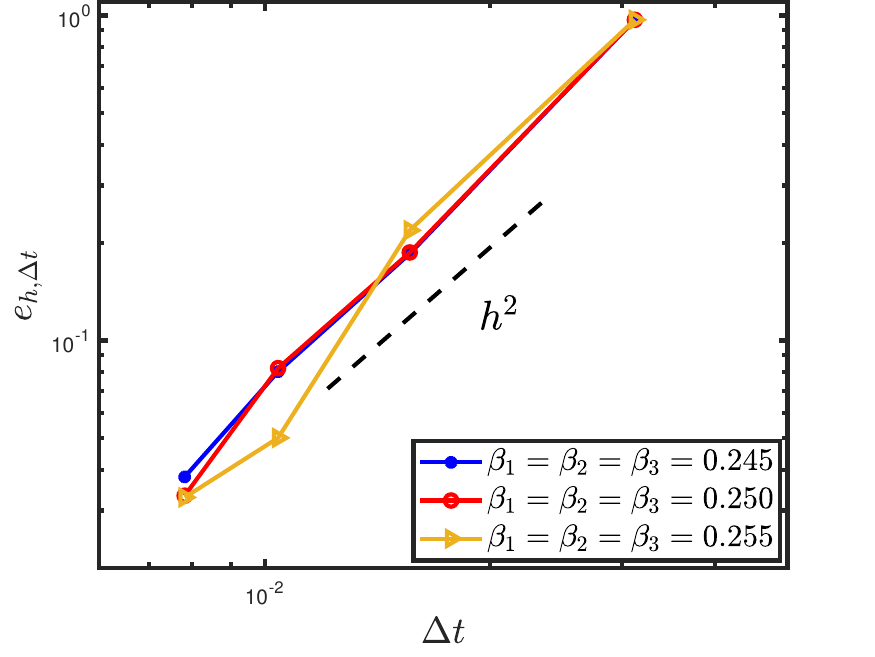}
\caption{Plot of numberical errors employing SP-PFEM at $\sigma_1 =\sigma_2 = -0.7$ (left panel) and $\sigma_1 =\sigma_2 = 0.7$ (right panel) for $2$-fold anisotropy. The parameters are selected as $\eta_1 =\eta_2 =\eta_3 = 100$, and $t_m = 2$. } 
\label{fig:3}
\end{figure}

\begin{figure}[!htp]
\centering
\includegraphics[width=0.49\textwidth]{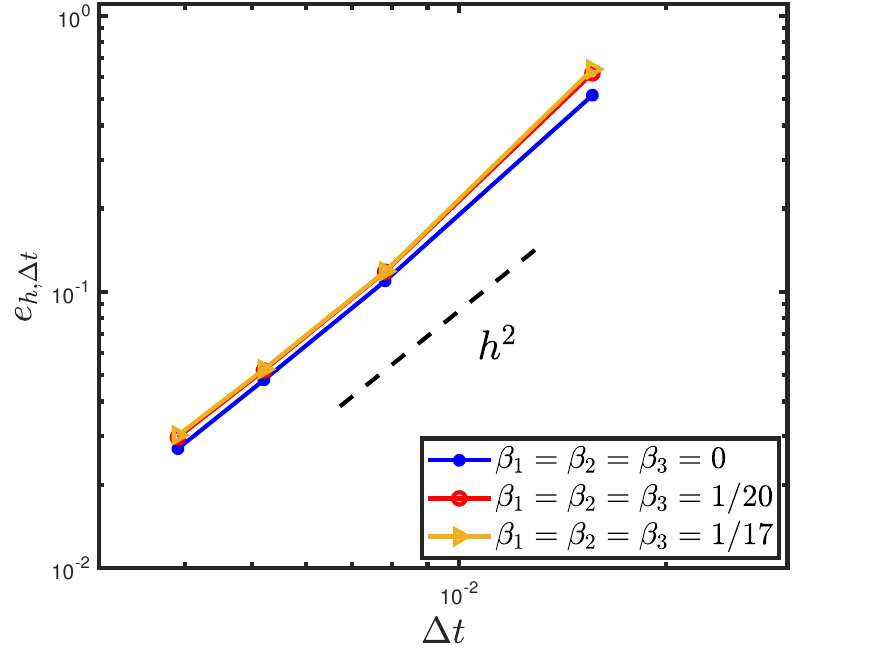}
\includegraphics[width=0.49\textwidth]{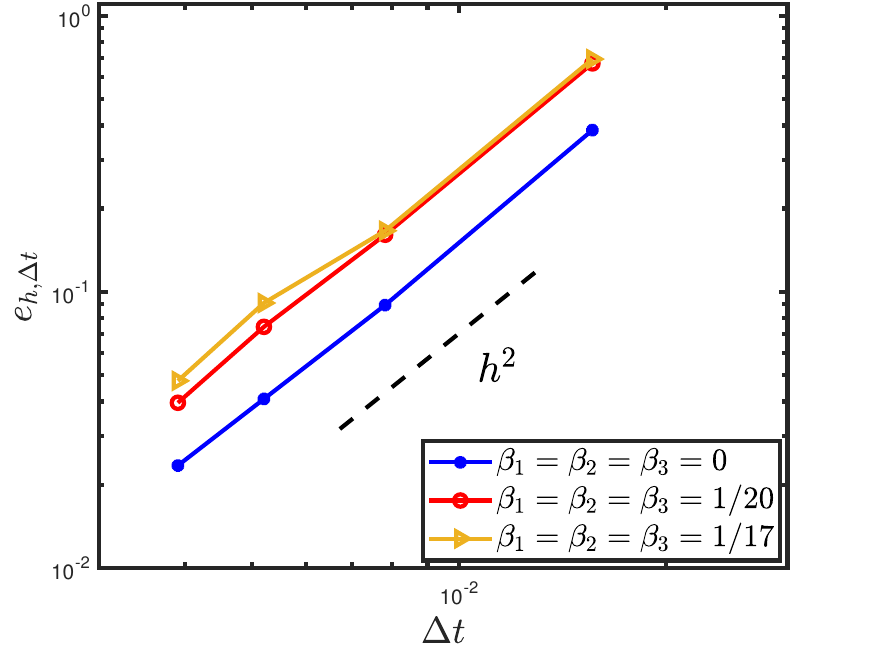}
\caption{Plot of numberical errors employing SP-PFEM at $\sigma_1 =\sigma_2 = -0.7$ (left panel) and $\sigma_1 =\sigma_2 = 0.7$ (right panel) for $4$-fold anisotropy. The parameters are selected as $\eta_1 =\eta_2 =\eta_3 = 100$, and $t_m = 1$. } 
\label{fig:4}
\end{figure}

\begin{figure}[!htp]
\centering
\includegraphics[width=0.49\textwidth]{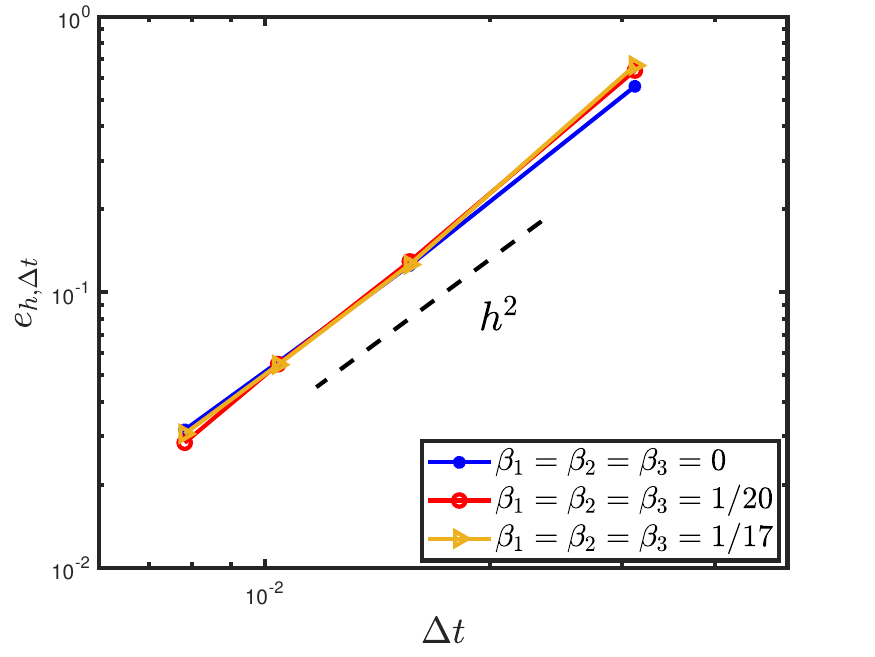}
\includegraphics[width=0.49\textwidth]{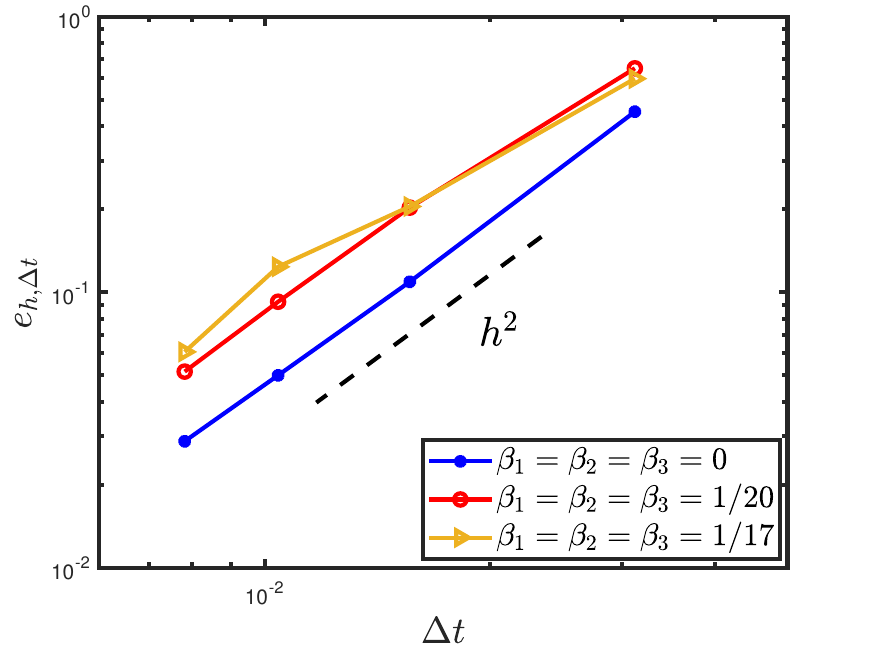}
\caption{Plot of numberical errors employing SP-PFEM at $\sigma_1 =\sigma_2 = -0.7$ (left panel) and $\sigma_1 =\sigma_2 = 0.7$ (right panel) for $4$-fold anisotropy. The parameters are selected as $\eta_1 =\eta_2 =\eta_3 = 100$, and $t_m = 2$. } 
\label{fig:5}
\end{figure}

\begin{figure}[!htp]
\centering
\includegraphics[width=0.45\textwidth]{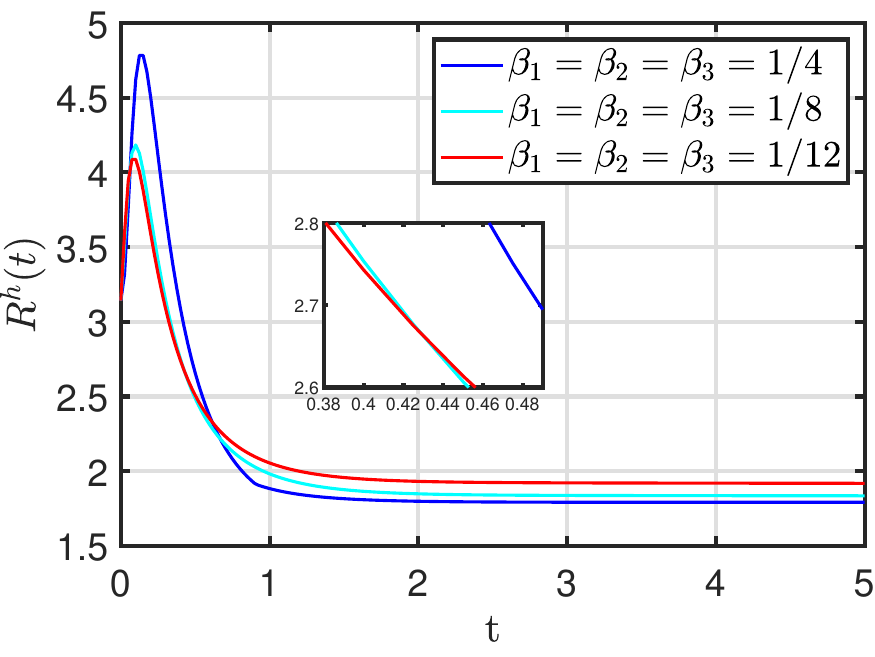}\hspace{0.5cm}
\includegraphics[width=0.45\textwidth]{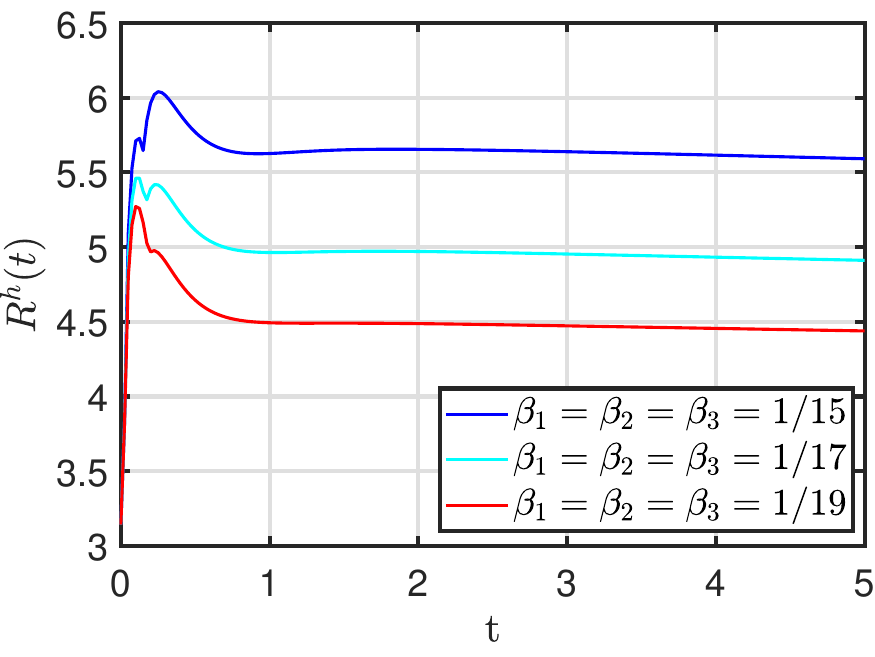}
\caption{Time evolution of the mesh ratio $R^h(t)$ for the two cases of $2$-fold (left panel) and $4$-fold (right panel) anisotropy. The parameters are chosen as $\ttau = 1/40$, $N_1 =N_2 =N_3 = 128$, $\eta_1 =\eta_2 =\eta_3 = 100$, and $\sigma_1 =\sigma_2 = -0.7$.} 
\label{fig:6}
\end{figure}

\begin{figure}[!htp]
\centering
\includegraphics[width=0.45\textwidth]{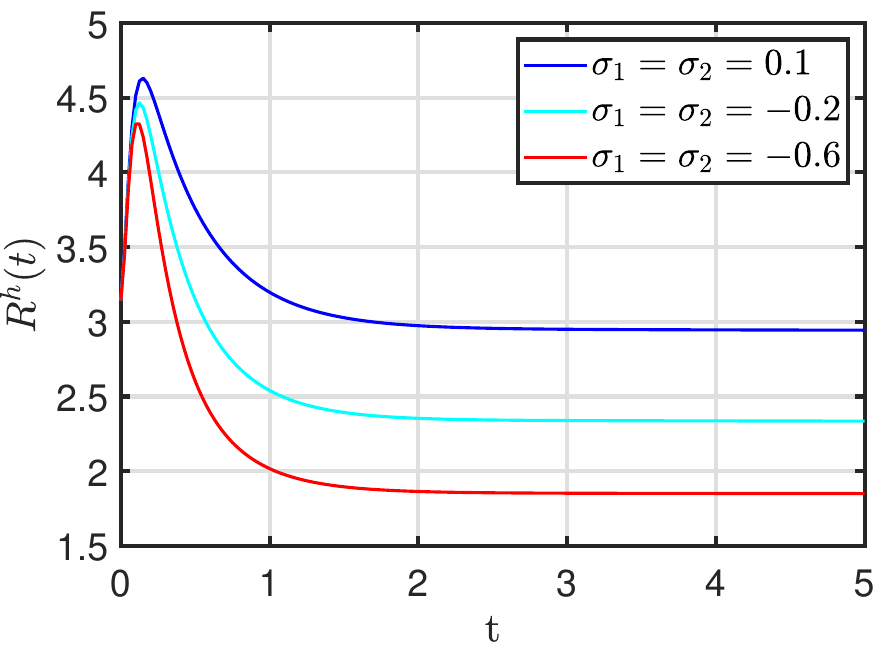}\hspace{0.5cm}
\includegraphics[width=0.45\textwidth]{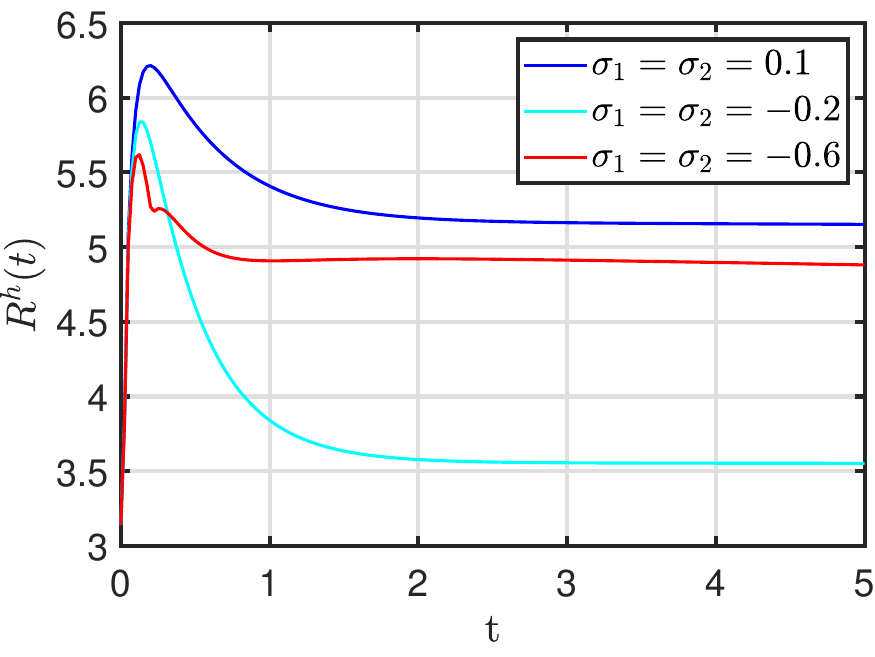}
\caption{Time evolution of the mesh ratio $R^h(t)$ for the two cases of $2$-fold (left panel) and $4$-fold (right panel) anisotropy. The parameters are chosen as $\ttau = 1/40$, $N_1 =N_2 =N_3 = 128$, $\eta_1 =\eta_2 =\eta_3 = 100$, and $\beta_1 =\beta_2=\beta_3 = 1/6$ (left panel), $\beta_1 =\beta_2=\beta_3 = 1/16$ (right panel).} 
\label{fig:7}
\end{figure}

\begin{example}\label{exa3}
(Area conservation $\&$ Energy stability)
In this example, we aim to verify the energy stability and area conservation properties of the SP-PFEM for the SSD of the double-bubble thin films. To achieve this, we first plot the total energy of the SP-PFEM by selecting different anisotropic strengths for the two types of anisotropic surface energy functions. As illustrated in Figure \ref{fig:8}, the results demonstrate the energy stability of the SP-PFEM under various anisotropic conditions.
In addition, to more effectively demonstrate the area conservation property of the SP-PFEM through numerical experiments, we compare its performance with that of the ES-PFEM. The time evolution of the relative area loss and total energy for both methods is illustrated in Figures \ref{fig:9} and \ref{fig:10}, corresponding to 2-fold and 4-fold anisotropy, respectively. The results show that the ES-PFEM experiences a significant relative area loss, reaching up to 7.2\% for 2-fold anisotropy and 4.2\% for 4-fold anisotropy. In contrast, the SP-PFEM exhibits exact area preservation, as expected. Furthermore, it is observed that the ES-PFEM attains a lower energy for the steady-state solution than the SP-PFEM, which may be attributed to area loss.
In our tests, we used a semi-elliptical shape as the initial configuration, consistent with Example \ref{exa1}.

\end{example}

\begin{figure}[!htp]
\centering
\includegraphics[width=0.45\textwidth]{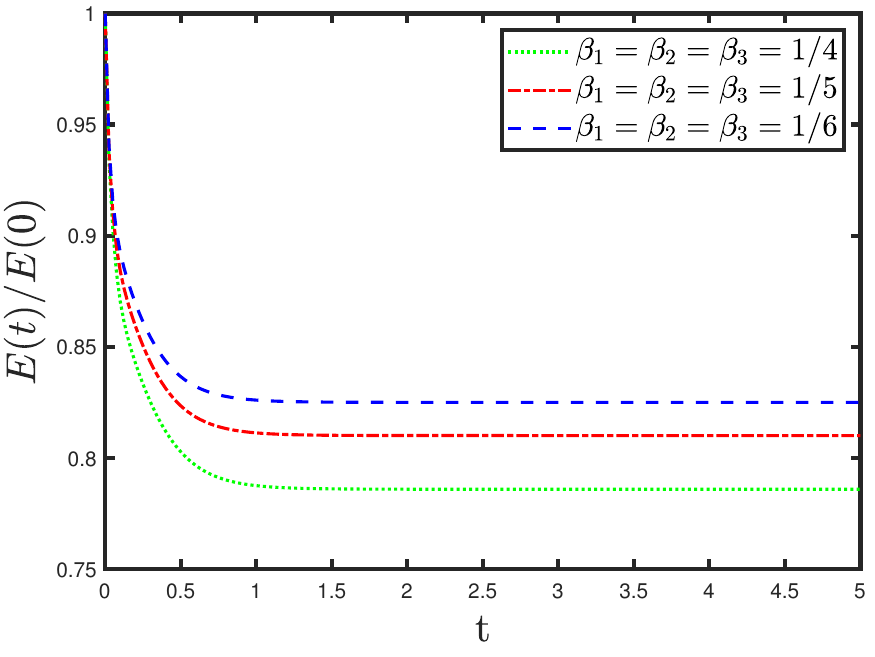}\hspace{0.5cm}
\includegraphics[width=0.45\textwidth]{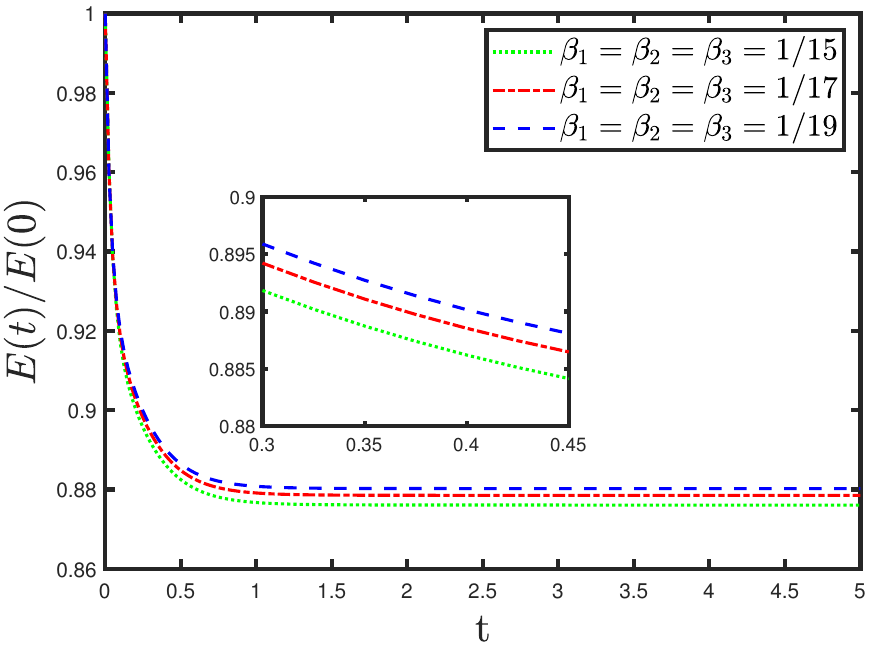}
\caption{The time history of the energy ${E}(t)/{E}(0)$ using SP-PFEM with $2$-fold (left panel) and $4$-fold (right panel) anisotropy. The parameters are chosen as $\ttau = 1/40$, $N_1 =N_2 =N_3 = 256$, $\eta_1 =\eta_2 =\eta_3 = 100$, and $\sigma_1 =\sigma_2 = -0.6$.}
\label{fig:8}
\end{figure}

\begin{figure}[!htp]
\centering
\includegraphics[width=0.3\textwidth]{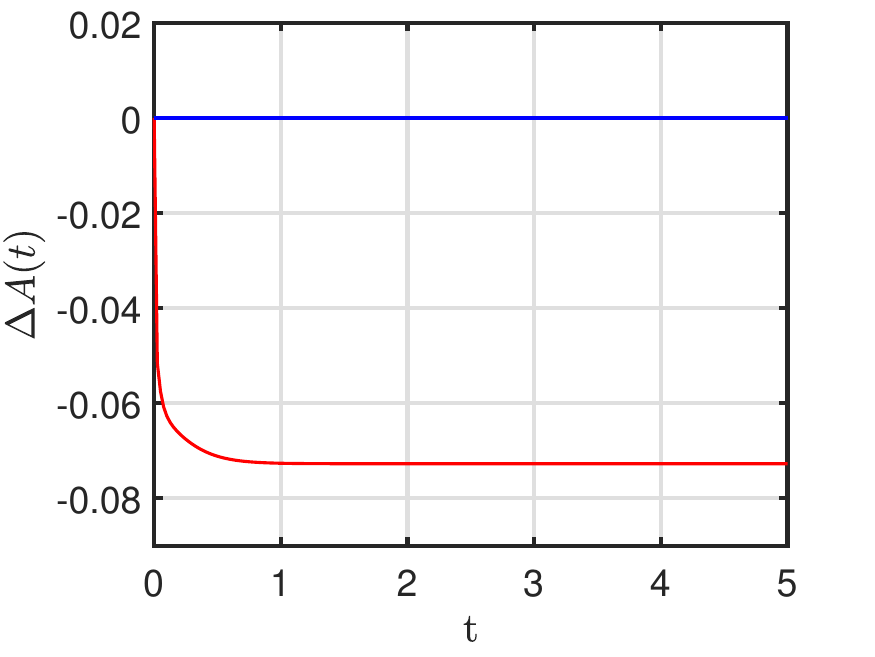}\hspace{0.5cm}
\includegraphics[width=0.3\textwidth]{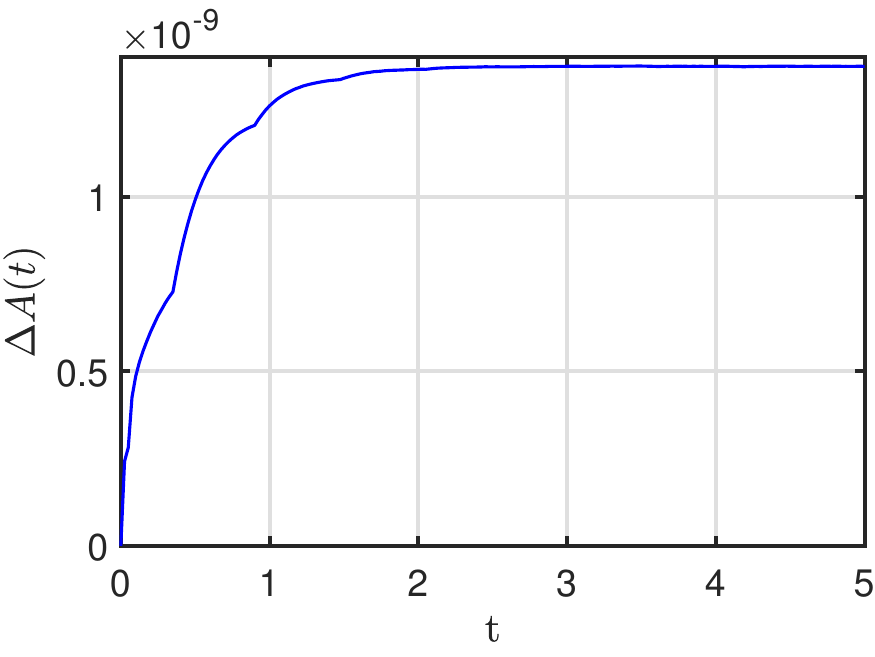}\hspace{0.5cm}
\includegraphics[width=0.3\textwidth]{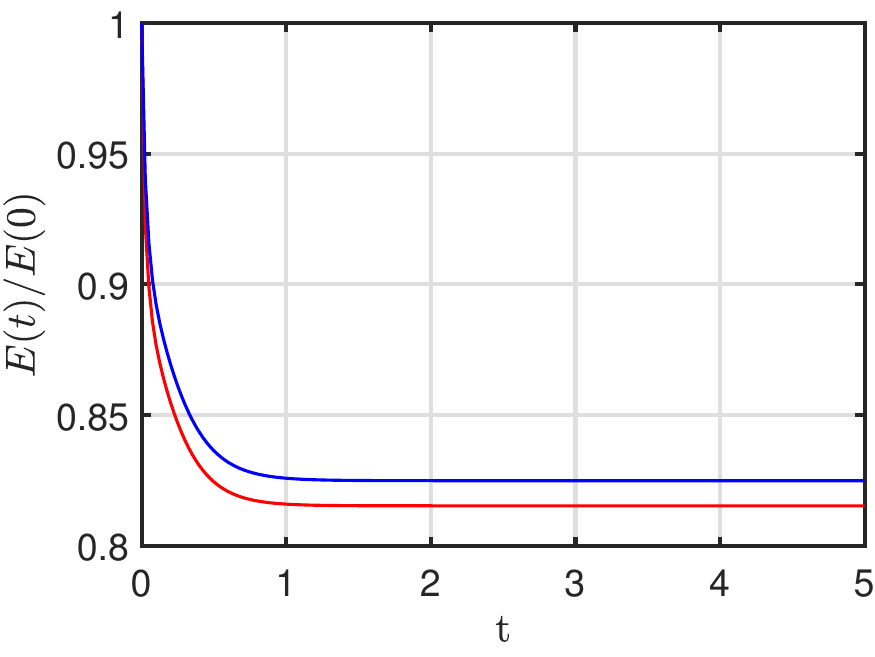}
\caption{The time history of the area loss $\Delta A(t)$ and the energy ${E}(t)/{E}(0)$. The blue line represents the results obtained using the SP-PFEM, while the red line represents the results obtained using the ES-PFEM for $2$-fold anisotropy. The degree of anisotropy is chosen as $\beta_1 =\beta_2 =\beta_3 = 1/6$. The parameters are chosen as $\ttau = 1/40$, $N_1 =N_2 =N_3 = 256$, $\eta_1 =\eta_2 =\eta_3 = 100$, and $\sigma_1 =\sigma_2 = -0.6$. }
\label{fig:9}
\end{figure}

\begin{figure}[!htp]
\centering
\includegraphics[width=0.3\textwidth]{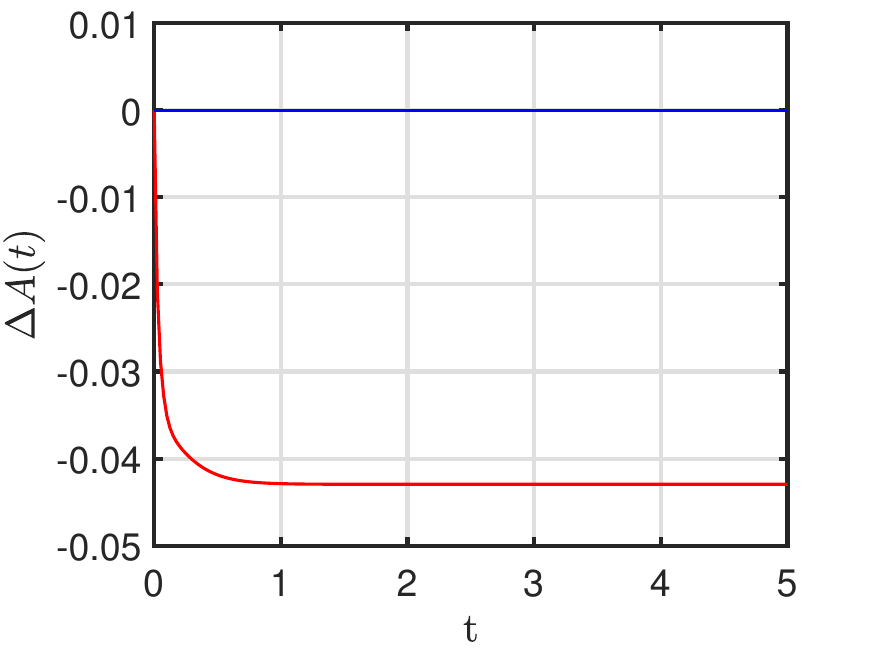}\hspace{0.5cm}
\includegraphics[width=0.3\textwidth]{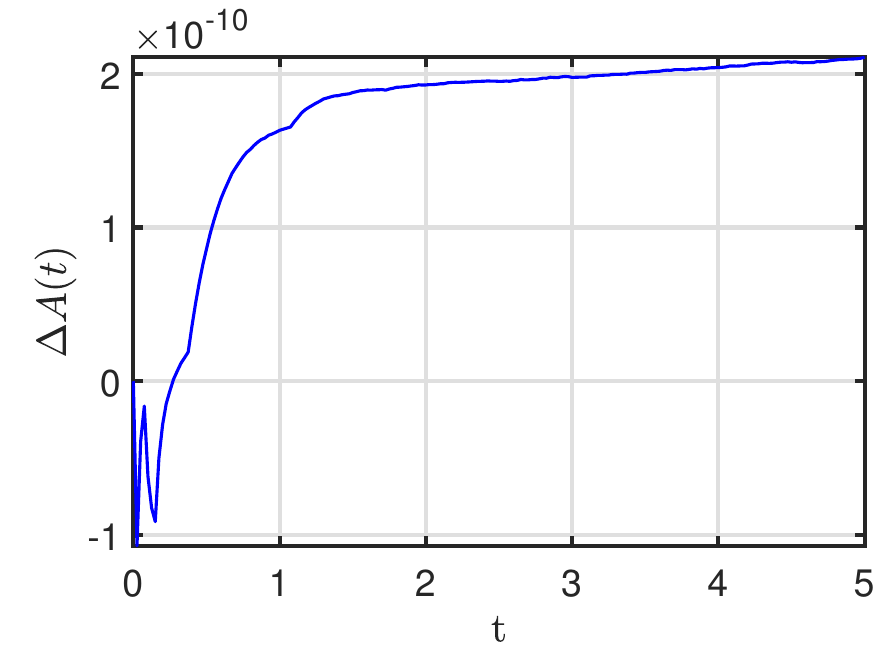}\hspace{0.5cm}
\includegraphics[width=0.3\textwidth]{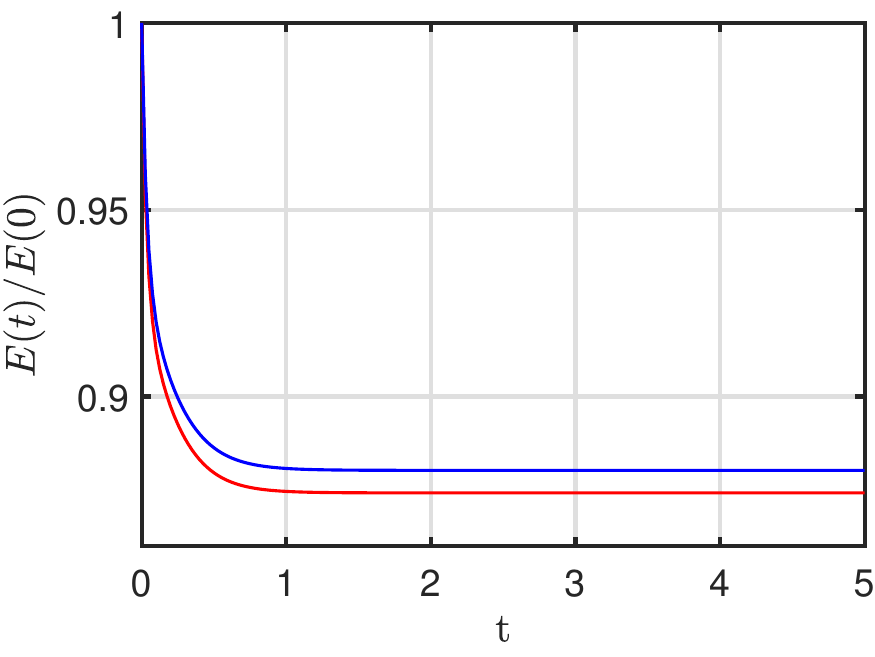}
\caption{The time history of the area loss $\Delta A(t)$ and the energy ${E}(t)/{E}(0)$. The blue line represents the results obtained using the SP-PFEM, while the red line represents the results obtained using the ES-PFEM for $4$-fold anisotropy. The degree of anisotropy is chosen as $\beta_1 =\beta_2 =\beta_3 = 1/19$. The parameters are chosen as $\ttau = 1/40$, $N_1 =N_2 =N_3 = 256$, $\eta_1 =\eta_2 =\eta_3 = 100$, and $\sigma_1 =\sigma_2 = -0.6$. }
\label{fig:10}
\end{figure}

\begin{example}\label{exa4}
(Equilibrium states)
In this example, we investigate the equilibrium states of double-bubble thin films during SSD, aiming to explore how different initial configurations evolve into equilibrium shapes under the influence of surface energy anisotropy and contact point conditions. We employ three commonly used initial shapes and demonstrate their evolution processes using the SP-PFEM, tracking their progression until equilibrium is reached.
The following are three selected initial triple curves: 
\begin{itemize}
    \item $\mat X_1^0(\rho_1) = \left( 2\cos\left(\pi-\frac{\pi\rho_1}{2}\right), 2\sin\left(\pi-\frac{\pi\rho_1}{2}\right)\right),\quad \mat X_2^0(\rho_2) = \left(2\cos\left(\frac{\pi\rho_2}{2}\right), 2\sin\left(\frac{\pi\rho_2}{2}\right)\right),\quad \mat X_3^0(\rho_3)=(0,2\rho_3)$.
    \item $\mat X_1^0(\rho_1) = \left(3\cos\left(\pi-\frac{\pi\rho_1}{2}\right), \frac{3}{2}\sin\left(\pi-\frac{\pi\rho_1}{2}\right)\right),\quad \mat X_2^0(\rho_2) = \left( 3\cos\left(\frac{\pi\rho_2}{2}\right), \frac{3}{2}\sin\left(\frac{\pi\rho_2}{2}\right)\right),\quad \mat X_3^0(\rho_3)=\left(0,\frac{3\rho_3}{2}\right)$.
 \item 
$\mat X_1^0(\rho_1) = \left(\left(2 + \cos\left(6\pi-3\pi\rho_1\right)\right)\cos\left(\pi-\frac{\pi\rho_1}{2}\right), (\left(2 + \cos\left(6\pi-3\pi\rho_1\right)\right)\sin\left(\pi-\frac{\pi\rho_1}{2}\right)\right),$\\
$\mat X_2^0(\rho_2) = \left(\left(2 + \cos\left(3\pi\rho_2\right)\right)\cos\left(\frac{\pi\rho_2}{2}\right), (\left(2 + \cos\left(3\pi\rho_2\right)\right)\sin\left(\frac{\pi\rho_2}{2}\right)\right),\quad\mat X_3^0(\rho_3)=(0,\rho_3).$
\end{itemize}
    
As illustrated in Figures \ref{fig:11}--\ref{fig:14}, different initial curves with the same anisotropy converge to the same equilibrium shape. This observation suggests that the equilibrium state of the system is predominantly governed by the surface energy anisotropy and contact angles, rather than the initial morphology.
Figure \ref{fig:15} presents the equilibrium shapes for the SSD problem with isotropic, weakly anisotropic, and strongly anisotropic surface energies. These figures demonstrate that the increase in the degree of anisotropy leads to more pronounced faceting and sharper corners in the equilibrium shapes. We also observe that, in the strongly anisotropic case, the equilibrium shape no longer exhibits symmetry about the third curve.
Figure \ref{fig:16} presents the equilibrium shapes obtained by employing non-uniform anisotropy parameters. From these figures, we also find that the positive or negative values of the contact angle parameters ($\sigma_1$ and $\sigma_2$) can lead to convex and concave structures, respectively.
\end{example}

\begin{figure}[!htp]
\centering
\includegraphics[width=0.3\textwidth]{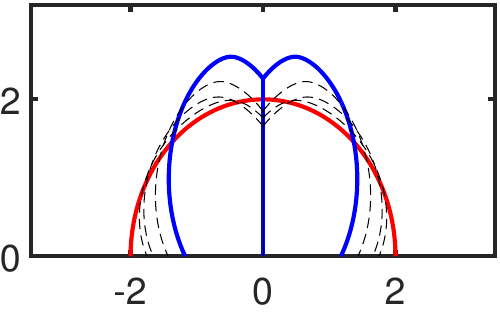}\hspace{0.5cm}
\includegraphics[width=0.3\textwidth]{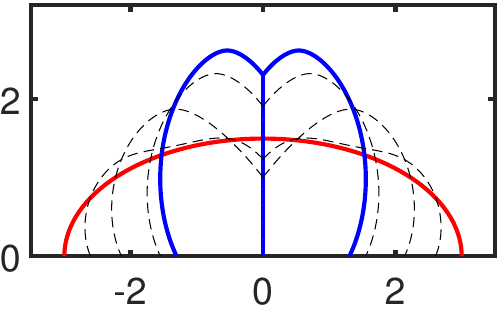}\hspace{0.5cm}
\includegraphics[width=0.3\textwidth]{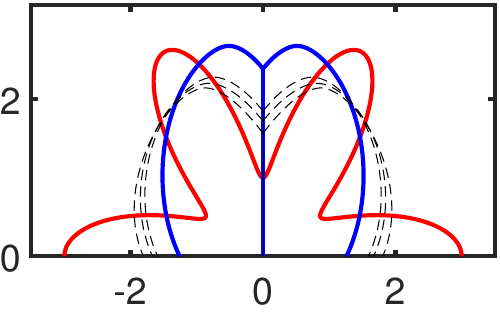}
\caption{The evolution of different initial curve (red line) to the equilibrium shape (blue line) with $2$-fold anisotropy. The degree of anisotropy are chosen as $\beta_1 =\beta_2 =\beta_3 = 1/4$. The other parameters are selected as $\ttau = 1/50$, $N_1 = N_2 = N_3 =  128$, $\eta_1 = \eta_2 = \eta_3 = 100$, and $\sigma_1 = \sigma_2 = -0.8$. }
\label{fig:11}
\end{figure}

\begin{figure}[!htp]
\centering
\includegraphics[width=0.3\textwidth]{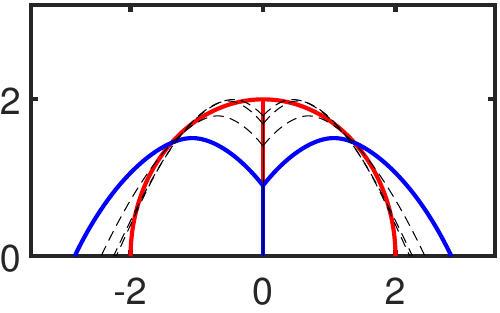}\hspace{0.5cm}
\includegraphics[width=0.3\textwidth]{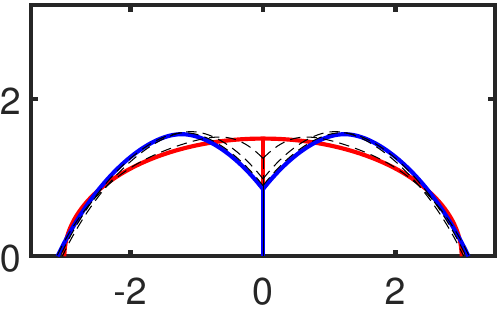}\hspace{0.5cm}
\includegraphics[width=0.3\textwidth]{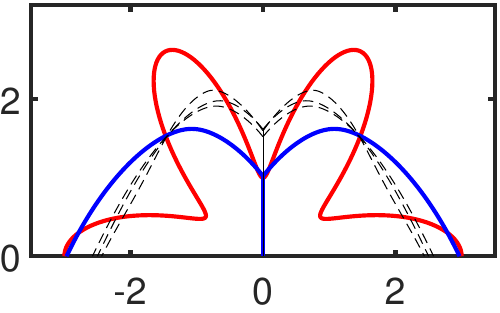}
\caption{The evolution of different initial curve (red line) to the equilibrium shape (blue line) with $2$-fold anisotropy. The degree of anisotropy are chosen as $\beta_1 =\beta_2 =\beta_3 = 1/4$. The other parameters are selected as $\ttau = 1/50$, $N_1 = N_2 = N_3 =  128$, $\eta_1 = \eta_2 = \eta_3 = 100$, and $\sigma_1 = \sigma_2 = 0.8$. }
\label{fig:12}
\end{figure}

\begin{figure}[!htp]
\centering
\includegraphics[width=0.3\textwidth]{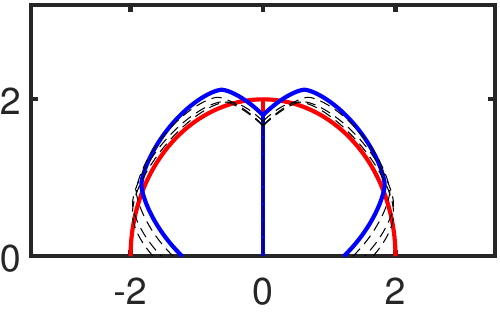}\hspace{0.5cm}
\includegraphics[width=0.3\textwidth]{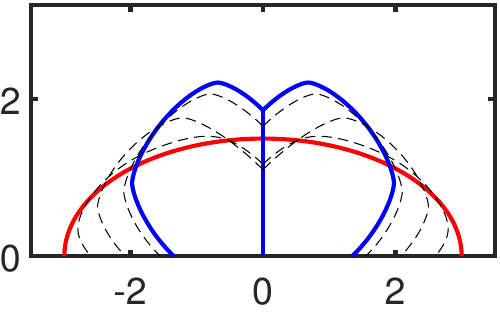}\hspace{0.5cm}
\includegraphics[width=0.3\textwidth]{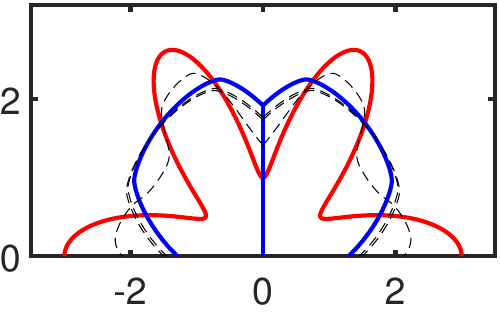}
\caption{The evolution of different initial curve (red line) to the equilibrium shape (blue line) with $4$-fold anisotropy. The degree of anisotropy are chosen as $\beta_1 =\beta_2 =\beta_3 = 1/17$. The other parameters are selected as $\ttau = 1/50$, $N_1 = N_2 = N_3 =  128$, $\eta_1 = \eta_2 = \eta_3 = 100$, and $\sigma_1 = \sigma_2 = -0.8$. }
\label{fig:13}
\end{figure}

\begin{figure}[!htp]
\centering
\includegraphics[width=0.3\textwidth]{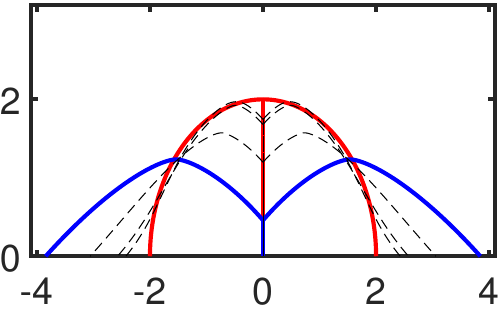}\hspace{0.5cm}
\includegraphics[width=0.3\textwidth]{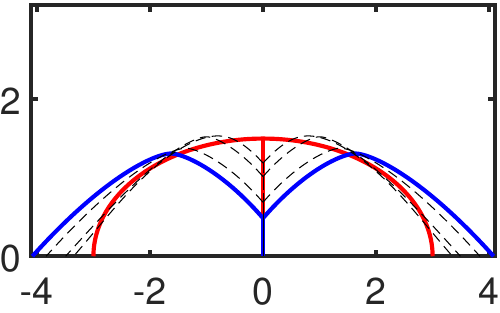}\hspace{0.5cm}
\includegraphics[width=0.3\textwidth]{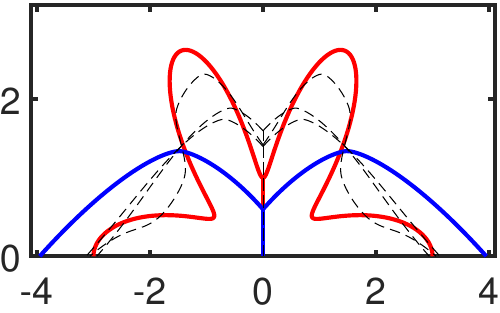}
\caption{The evolution of different initial curve (red line) to the equilibrium shape (blue line) with $4$-fold anisotropy. The degree of anisotropy are chosen as $\beta_1 =\beta_2 =\beta_3 = 1/17$. The other parameters are selected as $\ttau = 1/50$, $N_1 = N_2 = N_3 =  128$, $\eta_1 = \eta_2 = \eta_3 = 100$, and $\sigma_1 = \sigma_2 = 0.8$. }
\label{fig:14}
\end{figure}

\begin{figure}[!htp]
\centering
\includegraphics[width=0.45\textwidth]{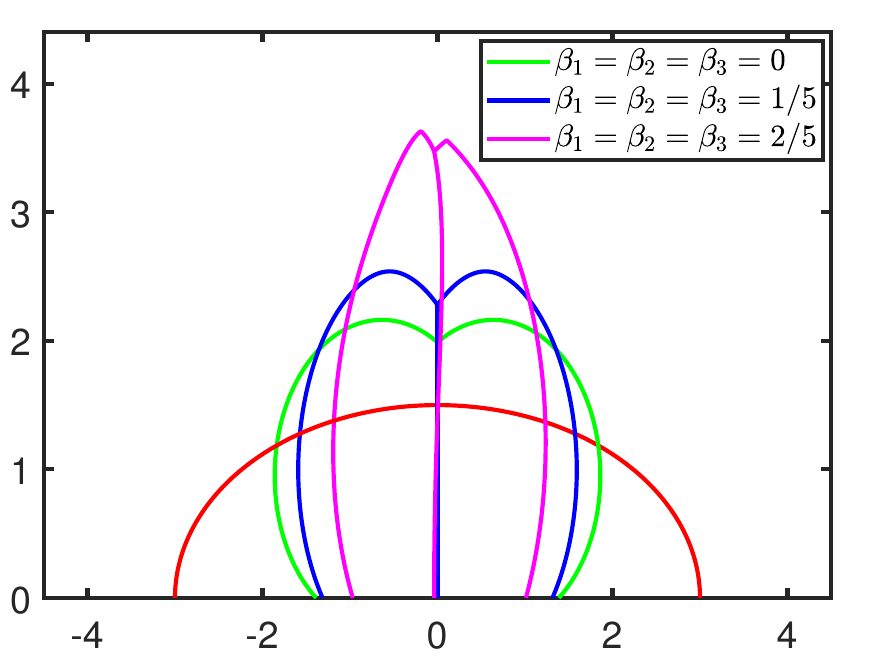}\hspace{0.5cm}
\includegraphics[width=0.45\textwidth]{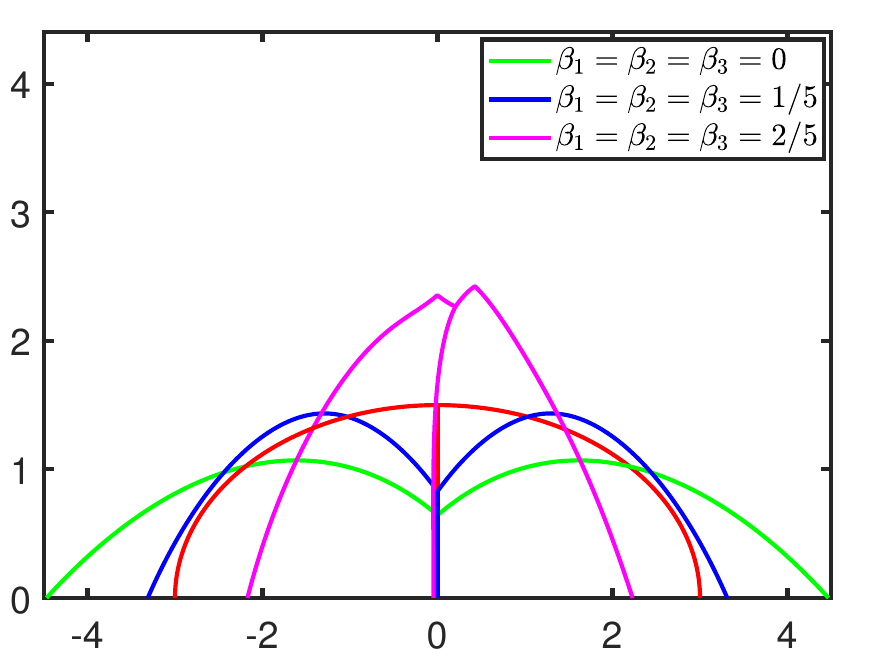}\hspace{0.5cm}
\includegraphics[width=0.45\textwidth]{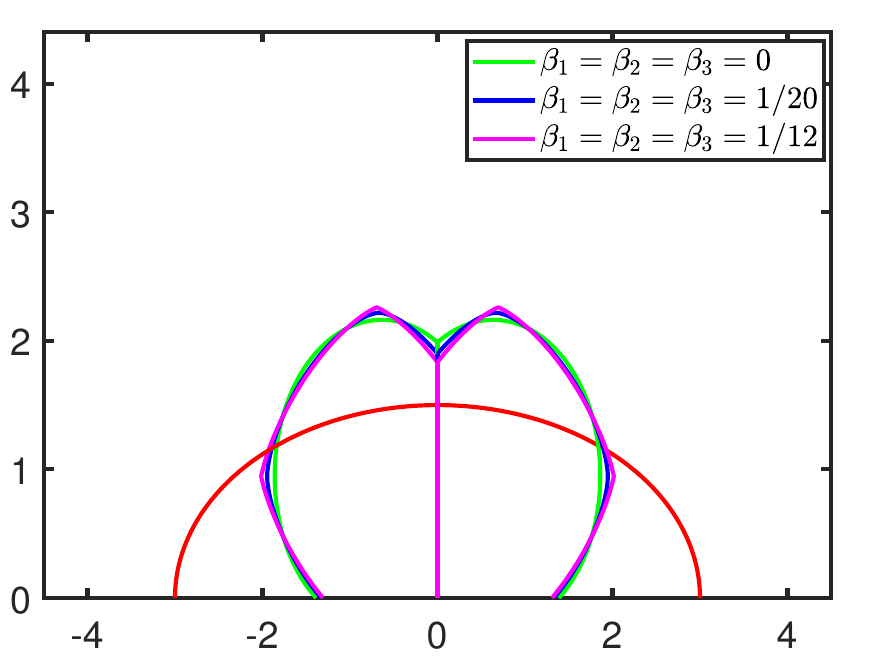}\hspace{0.5cm}
\includegraphics[width=0.45\textwidth]{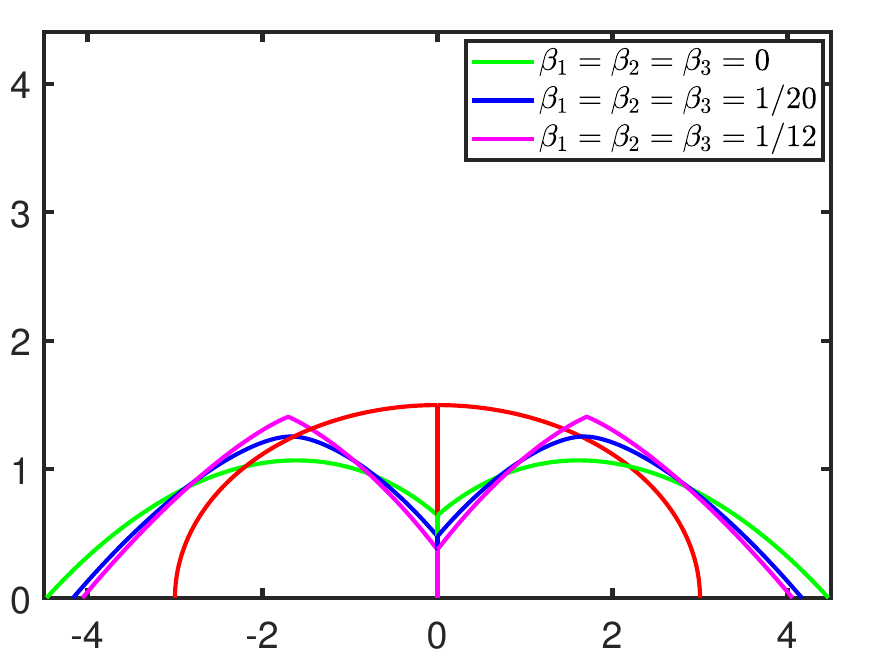}
\caption{The evolution of different initial curve (red line) to the equilibrium shape at $\sigma_1 =\sigma_2 = -0.8$ (top left panel) and $\sigma_1 =\sigma_2 = 0.8$ (top right panel) for $2$-fold anisotropy and at $\sigma_1 =\sigma_2 = -0.8$ (bottom left panel) and $\sigma_1 =\sigma_2 = 0.8$ (bottom right panel) for $4$-fold anisotropy. The other parameters are selected as $\ttau = 7/100$, $N_1 = N_2 = N_3 =  128$, $\eta_1 = \eta_2 = \eta_3 = 100$. }
\label{fig:15}
\end{figure}

\begin{figure}[!htp]
\centering
\includegraphics[width=0.45\textwidth]{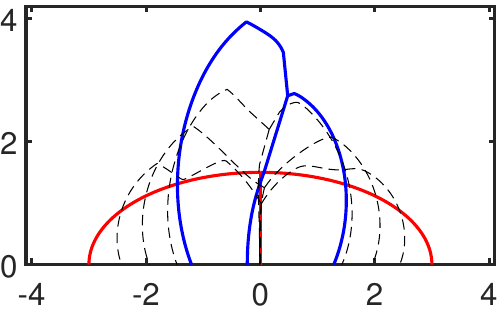}\hspace{0.5cm}
\includegraphics[width=0.45\textwidth]{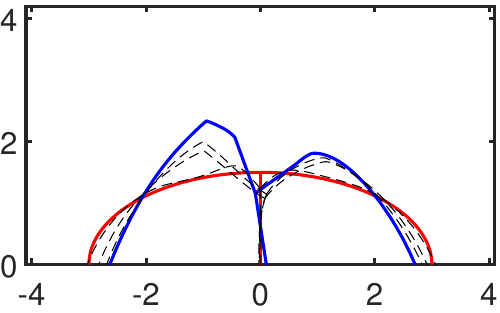}\hspace{0.5cm}
\includegraphics[width=0.45\textwidth]{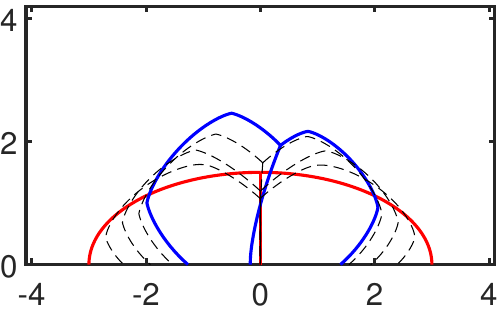}\hspace{0.5cm}
\includegraphics[width=0.45\textwidth]{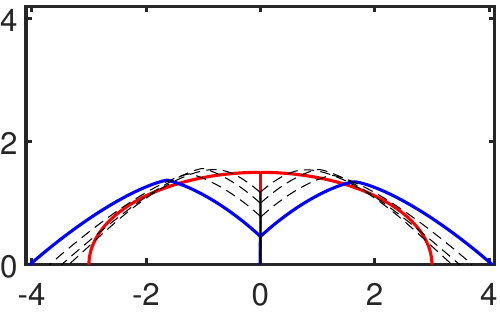}
\caption{The evolution of different initial curve (red line) to the equilibrium shape at $\sigma_1 =\sigma_2 = -0.8$ (top left panel) and $\sigma_1 =\sigma_2 = 0.8$ (top right panel) for $2$-fold anisotropy with anisotropy parameters $\beta_1 = 2/5$, $\beta_2 =1/3$, $\beta_3 = 1/4$ and at $\sigma_1 =\sigma_2 = -0.8$ (bottom left panel) and $\sigma_1 =\sigma_2 = 0.8$ (bottom right panel) for $4$-fold anisotropy with anisotropy parameters $\beta_1 = 1/14$, $\beta_2 =1/15$, $\beta_3 = 1/16$. The other parameters are selected as $\ttau = 1/50$, $N_1 = N_2 = N_3 =  128$, $\eta_1 = \eta_2 = \eta_3 = 100$. }
\label{fig:16}
\end{figure}

\begin{example}\label{exa5}
(Pinch-off events)
In this experiment, we simulate the pinch-off phenomenon that occurs during the evolution of thin film when the film becomes very flat.
As illustrated in Figures \ref{fig:17}-\ref{fig:18}, the process of SSD occurs rapidly, initially leading to the formation of ridges along the edges of the island. Subsequently, several valleys begin to develop. Over time, these ridges and valleys become more pronounced, and the valleys gradually merge near the center of the island. Eventually, the central valley reaches the substrate. At this critical point, the grid is reinitialized and the region where the film contacts the substrate is divided into several sections. 
Here we can find that the point of the first pinch-off is exactly the starting point $C$ of the third curve.
These sections then continue to evolve from their initial configurations, each undergoing further pinch-off events. This process ultimately reaches a steady state, resulting in the formation of multiple smaller islands. Through comparison, we also observe that the anisotropy strength affects the occurrence time of the pinch-off phenomenon.
In this example, we choose the initial data as
\begin{equation*}
    \mat X_1^0(\rho_1) = \left( 80\cos\left(\pi-\frac{\pi\rho_1}{2}\right), \frac{1}{10}\sin\left(\pi-\frac{\pi\rho_1}{2}\right)\right),  \quad \mat X_2^0(\rho_2) = \left(80\cos\left(\frac{\pi\rho_2}{2}\right),\frac{1}{10}\sin\left(\frac{\pi\rho_2}{2}\right)\right),\quad  \mat X_3^0(\rho_3)=\left(0,\frac{\rho_3}{10}\right).
\end{equation*}
\end{example}

\begin{figure}[!htp]
\centering
\includegraphics[width=0.8\textwidth]{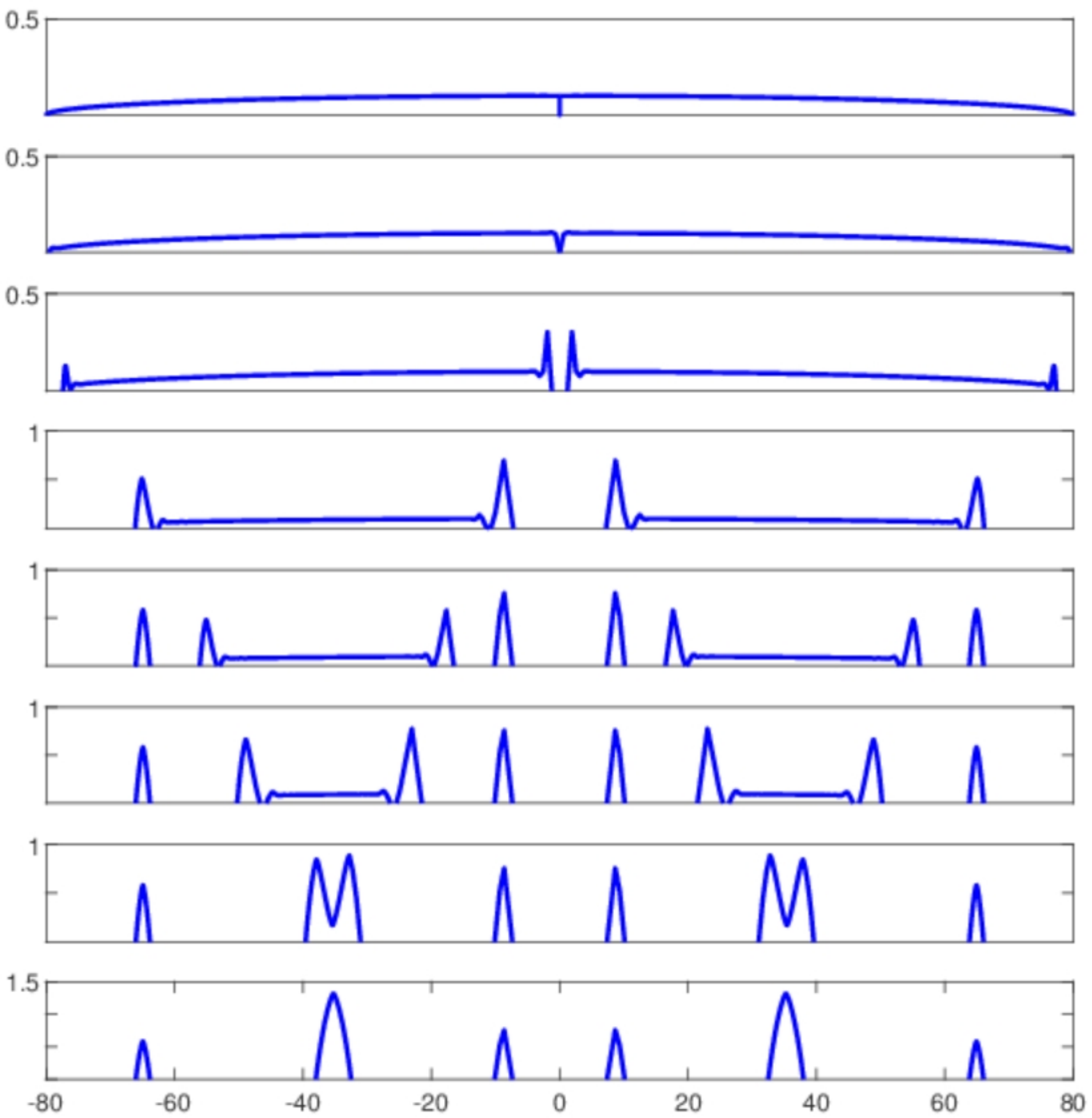}
\caption{Several snapshots in the evolution of a long, flat initial curve with $4$-fold anisotropy at times $t = 0, 0.020, 0.120, 5.120, 7.615, 16.615, 60.115, 67.625$ (from top to bottom). The degree of anisotropy are chosen as $\beta_1 =\beta_2=\beta_3 = 1/15$. The other parameters are selected as $\ttau = 1/200$, $N_1 = N_2 = N_3 =  200$, $\eta_1 = \eta_2 = \eta_3 = 100$, and $\sigma_1 = \sigma_2 = 0.8$. }
\label{fig:17}
\end{figure}

\begin{figure}[!htp]
\centering
\includegraphics[width=0.8\textwidth]{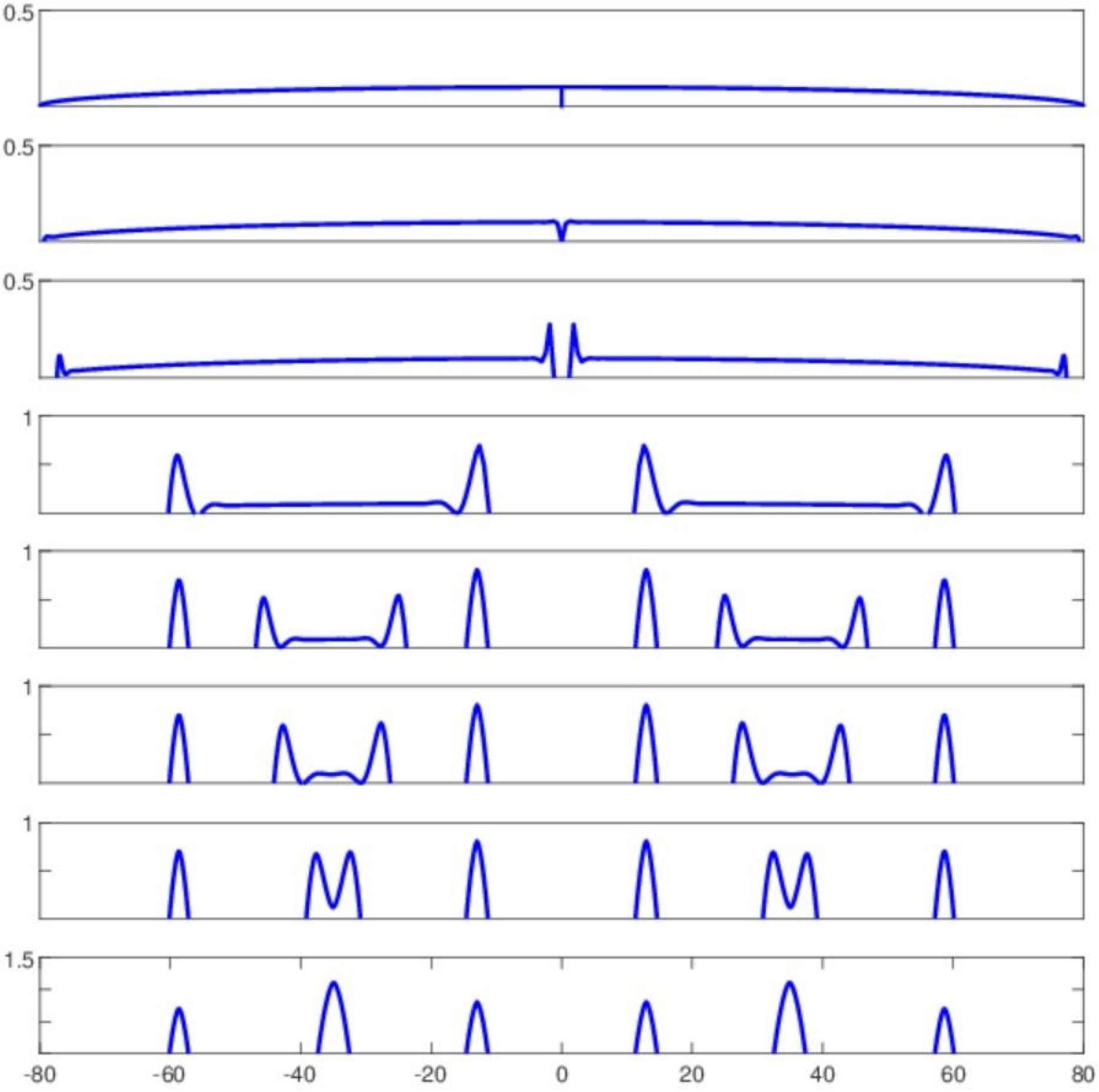}
\caption{Several snapshots in the evolution of a long, flat initial curve with $4$-fold anisotropy at times $t=0, 0.025, 0.125, 14.330, 19.325, 24.325, 39.325, 45.125$ (from top to bottom). The degree of anisotropy are chosen as $\beta_1 =\beta_2=\beta_3 = 1/19$. The other parameters are selected as $\ttau = 1/200$, $N_1 = N_2 = N_3 =  200$, $\eta_1 = \eta_2 = \eta_3 = 100$, and $\sigma_1 = \sigma_2 = 0.8$. }
\label{fig:18}
\end{figure}

\section{Conclusions}\label{sec6}
In this study, we develop a comprehensive framework for modeling SSD of double-bubble thin films with both weakly and strongly anisotropic surface energies. By combining thermodynamic variation with a smooth vector-field perturbation method, we rigorously derive a sharp-interface model for the system. Additionally, we construct a structure-preserving PFEM that ensures area conservation and energy stability for the proposed sharp-interface model. Extensive numerical experiments are conducted to validate the proposed method, demonstrating its convergence, structure-preserving properties, and superior mesh quality.
Furthermore, we investigate several specific evolution processes in the SSD of double-bubble thin films, including the evolution of small films and the pinch-off of long islands. These findings offer valuable insights into the underlying mechanisms of SSD in multiple thin films and highlight the versatility and robustness of our proposed method.

\bibliographystyle{elsarticle-num}
\bibliography{thebib}

\begin{thebibliography}{10}
\expandafter\ifx\csname url\endcsname\relax
  \def\url#1{\texttt{#1}}\fi
\expandafter\ifx\csname urlprefix\endcsname\relax\def\urlprefix{URL }\fi
\expandafter\ifx\csname href\endcsname\relax
  \def\href#1#2{#2} \def\path#1{#1}\fi

\bibitem{Kan05}
W.~Kan, H.~Wong, Fingering instability of a retracting solid film edge, J.
  Appl. Phys. 97~(4) (2005) 043515.

\bibitem{Ye10b}
J.~Ye, C.~V. Thompson, Regular pattern formation through the retraction and
  pinch-off of edges during solid-state dewetting of patterned single crystal
  films, Phys. Rev. B 82~(19) (2010) 193408.

\bibitem{Ye11a}
J.~Ye, C.~V. Thompson, Anisotropic edge retraction and hole growth during
  solid-state dewetting of single crystal nickel thin films, Acta Mater. 59~(2)
  (2011) 582--589.

\bibitem{Ye11b}
J.~Ye, C.~V. Thompson, Templated solid-state dewetting to controllably produce
  complex patterns, Adv. Mater. 23~(13) (2011) 1567--1571.

\bibitem{Wong00}
H.~Wong, P.~Voorhees, M.~Miksis, S.~Davis, Periodic mass shedding of a
  retracting solid film step, Acta Mater. 48~(8) (2000) 1719--1728.

\bibitem{dornel2006surface}
E.~Dornel, J.~Barbe, F.~De~Cr{\'e}cy, G.~Lacolle, J.~Eymery, Surface diffusion
  dewetting of thin solid films: Numerical method and application to
  {S}i/{S}io2, Phys. Rev. B 73~(11) (2006) 115427.

\bibitem{hyun2013quantitative}
G.~Hyun~Kim, R.~V. Zucker, J.~Ye, W.~Craig~Carter, C.~V. Thompson, Quantitative
  analysis of anisotropic edge retraction by solid-state dewetting of thin
  single crystal films, J. Appl. Phys. 113~(4) (2013) 043512.

\bibitem{Jiran90}
E.~Jiran, C.~Thompson, Capillary instabilities in thin films, J. Electron.
  Mater. 19~(11) (1990) 1153--1160.

\bibitem{Jiran92}
E.~Jiran, C.~Thompson, Capillary instabilities in thin, continuous films, Thin
  Solid Films. 208~(1) (1992) 23--28.

\bibitem{Ye10a}
J.~Ye, C.~V. Thompson, Mechanisms of complex morphological evolution during
  solid-state dewetting of single-crystal nickel thin films, Appl. Phys. Lett.
  97~(7) (2010) 071904.

\bibitem{Jiang12}
W.~Jiang, W.~Bao, C.~V. Thompson, D.~J. Srolovitz, Phase field approach for
  simulating solid-state dewetting problems, Acta Mater. 60~(15) (2012)
  5578--5592.

\bibitem{Kim15}
G.~H. Kim, C.~V. Thompson, Effect of surface energy anisotropy on
  {Rayleigh-like} solid-state dewetting and nanowire stability, Acta Mater. 84
  (2015) 190--201.

\bibitem{Mizsei93}
J.~Mizsei, Activating technology of {SnO$_2$} layers by metal particles from
  ultrathin metal films, Sensor Actuat B-Chem. 16~(1) (1993) 328--333.

\bibitem{Armelao06}
L.~Armelao, D.~Barreca, G.~Bottaro, A.~Gasparotto, S.~Gross, C.~Maragno,
  E.~Tondello, Recent trends on nanocomposites based on {Cu, Ag and Au}
  clusters: A closer look, Coord. Chem. Rev. 250~(11) (2006) 1294--1314.

\bibitem{Schmidt09}
V.~Schmidt, J.~V. Wittemann, S.~Senz, U.~G{\"o}sele, Silicon nanowires: a
  review on aspects of their growth and their electrical properties, Adv.
  Mater. 21~(25-26) (2009) 2681--2702.

\bibitem{Amram12}
D.~Amram, L.~Klinger, E.~Rabkin, Anisotropic hole growth during solid-state
  dewetting of single-crystal {Au--Fe} thin films, Acta Mater. 60~(6-7) (2012)
  3047--3056.

\bibitem{Rabkin14}
E.~Rabkin, D.~Amram, E.~Alster, Solid state dewetting and stress relaxation in
  a thin single crystalline {Ni} film on sapphire, Acta Mater. 74 (2014)
  30--38.

\bibitem{Herz216}
A.~Herz, A.~Franz, F.~Theska, M.~Hentschel, T.~Kups, D.~Wang, P.~Schaaf,
  Solid-state dewetting of single-and bilayer {Au-W} thin films: Unraveling the
  role of individual layer thickness, stacking sequence and oxidation on
  morphology evolution, AIP Adv. 6~(3) (2016) 035109.

\bibitem{Naffouti16}
M.~Naffouti, T.~David, A.~Benkouider, L.~Favre, A.~Delobbe, A.~Ronda,
  I.~Berbezier, M.~Abbarchi, Templated solid-state dewetting of thin silicon
  films, Small 12~(44) (2016) 6115--6123.

\bibitem{Naffouti17}
M.~Naffouti, R.~Backofen, M.~Salvalaglio, T.~Bottein, M.~Lodari, A.~Voigt,
  T.~David, A.~Benkouider, I.~Fraj, L.~Favre, et~al., Complex dewetting
  scenarios of ultrathin silicon films for large-scale nanoarchitectures, Sci.
  Adv. 3~(11) (2017) eaao1472.

\bibitem{Kovalenko17}
O.~Kovalenko, S.~Szab{\'o}, L.~Klinger, E.~Rabkin, Solid state dewetting of
  polycrystalline {Mo} film on sapphire, Acta Mater. 139 (2017) 51--61.

\bibitem{Dornel06}
E.~Dornel, J.~Barbe, F.~De~Cr{\'e}cy, G.~Lacolle, J.~Eymery, Surface diffusion
  dewetting of thin solid films: Numerical method and application to
  {Si/SiO$_2$}, Phys. Rev. B 73~(11) (2006) 115427.

\bibitem{Jiang16}
W.~Jiang, Y.~Wang, Q.~Zhao, D.~J. Srolovitz, W.~Bao, Solid-state dewetting and
  island morphologies in strongly anisotropic materials, Scr. Mater. 115 (2016)
  123--127.

\bibitem{Srolovitz86a}
D.~J. Srolovitz, S.~A. Safran, Capillary instabilities in thin films: {I.}
  {Energetics}, J. Appl. Phys. 60~(1) (1986) 247--254.

\bibitem{Srolovitz86}
D.~J. Srolovitz, S.~A. Safran, Capillary instabilities in thin films: {II.}
  {Kinetics}, J. Appl. Phys. 60~(1) (1986) 255--260.

\bibitem{Wang15}
Y.~Wang, W.~Jiang, W.~Bao, D.~J. Srolovitz, Sharp interface model for
  solid-state dewetting problems with weakly anisotropic surface energies,
  Phys. Rev. B 91 (2015) 045303.

\bibitem{baojcm2022}
W.~Bao, Q.~Zhao, An energy-stable parametric finite element method for
  simulating solid-state dewetting problems in three dimensions, J. Comput.
  Math. 41~(4) (2023) 771--796.

\bibitem{Bao17}
W.~Bao, W.~Jiang, Y.~Wang, Q.~Zhao, A parametric finite element method for
  solid-state dewetting problems with anisotropic surface energies, J. Comput.
  Phys. 330 (2017) 380--400.

\bibitem{Bao17b}
W.~Bao, W.~Jiang, D.~J. Srolovitz, Y.~Wang, Stable equilibria of anisotropic
  particles on substrates: a generalized {Winterbottom} construction, SIAM J.
  Appl. Math. 77~(6) (2017) 2093--2118.

\bibitem{Zucker16}
R.~V. Zucker, G.~H. Kim, J.~Ye, W.~C. Carter, C.~V. Thompson, The mechanism of
  corner instabilities in single-crystal thin films during dewetting, J. Appl.
  Phys. 119~(12) (2016) 125306.

\bibitem{Zhao20}
Q.~Zhao, W.~Jiang, W.~Bao, An energy-stable parametric finite element method
  for simulating solid-state dewetting, IMA J. Numer. Anal. 41~(3) (2021)
  2026--2055.

\bibitem{Thompson12}
C.~V. Thompson, Solid-state dewetting of thin films, Annu. Rev. Mater. Res. 42
  (2012) 399--434.

\bibitem{Leroy16}
F.~Leroy, F.~Cheynis, Y.~Almadori, S.~Curiotto, M.~Trautmann, J.~Barb{\'e},
  P.~M{\"u}ller, et~al., How to control solid state dewetting: A short review,
  Surf. Sci. Rep. 71~(2) (2016) 391--409.

\bibitem{Pierre09b}
O.~Pierre-Louis, A.~Chame, Y.~Saito, Dewetting of ultrathin solid films, Phys.
  Rev. Lett. 103~(19) (2009) 195501.

\bibitem{Dufay11}
M.~Dufay, O.~Pierre-Louis, Anisotropy and coarsening in the instability of
  solid dewetting fronts, Phys. Rev. Lett. 106~(10) (2011) 105506.

\bibitem{Klinger11shape}
L.~Klinger, E.~Rabkin, Shape evolution by surface and interface diffusion with
  rigid body rotations, Acta Mater. 59~(17) (2011) 6691--6699.

\bibitem{Zucker13}
R.~V. Zucker, G.~H. Kim, W.~C. Carter, C.~V. Thompson, A model for solid-state
  dewetting of a fully-faceted thin film, C. R. Phys. 14~(7) (2013) 564--577.

\bibitem{Jiang19a}
W.~Jiang, Q.~Zhao, Sharp-interface approach for simulating solid-state
  dewetting in two dimensions: a {Cahn-Hoffman} $\boldsymbol{\xi}$-vector
  formulation, Physica D 390 (2019) 69--83.

\bibitem{Zhao19b}
Q.~Zhao, W.~Jiang, W.~Bao, A parametric finite element method for solid-state
  dewetting problems in three dimensions, SIAM J. Sci. Comput. 42~(1) (2020)
  B327--B352.

\bibitem{amilibia2001existence}
A.~M. Amilibia, Existence and uniqueness of standard bubble clusters of given
  volumes in $\mathbb{R}^n$, Asian J. Math. 5~(1) (2001) 25--31.

\bibitem{foisy1993standard}
J.~Foisy, M.~Alfaro~Garcia, J.~Brock, N.~Hodges, J.~Zimba, The standard double
  soap bubble in ${R}^2$ uniquely minimizes perimeter, Pac. J. Math. 159~(1)
  (1993) 47--59.

\bibitem{wichiramala2004proof}
W.~Wichiramala, Proof of the planar triple bubble conjecture, J. für die Reine
  und Angew. Math. 2004~(567) (2004) 1--49.

\bibitem{paolini2020quadruple}
E.~Paolini, V.~M. Tortorelli, The quadruple planar bubble enclosing equal areas
  is symmetric, Calc. Var. Partial Differ. 59 (2020) 1--9.

\bibitem{hutchings2002proof}
M.~Hutchings, F.~Morgan, M.~Ritor{\'e}, A.~Ros, Proof of the double bubble
  conjecture, Ann. Math. (2002) 459--489.

\bibitem{sullivan1996open}
J.~M. Sullivan, F.~Morgan, Open problems in soap bubble geometry, Int. J. Math.
  7~(06) (1996) 833--842.

\bibitem{taylor1976structure}
J.~E. Taylor, The structure of singularities in soap-bubble-like and
  soap-film-like minimal surfaces, Ann. Math. 103~(3) (1976) 489--539.

\bibitem{morgan2007colloquium}
F.~Morgan, Colloquium: Soap bubble clusters, Rev. Mod. Phys. 79~(3) (2007)
  821--827.

\bibitem{morgan1998wulff}
F.~Morgan, C.~French, S.~Greenleaf, Wulff clusters in ${R}^2$, J. Geom. Anal. 8
  (1998) 97--115.

\bibitem{wecht2000double}
B.~Wecht, M.~Barber, J.~Tice, Double crystals, Acta Cryst. 56~(1) (2000)
  92--95.

\bibitem{Hoffman72}
D.~W. Hoffman, J.~W. Cahn, A vector thermodynamics for anisotropic surfaces: I.
  fundamentals and application to plane surface junctions, Surf. Sci. 31 (1972)
  368--388.

\bibitem{Cahn74}
J.~Cahn, D.~Hoffman, A vector thermodynamics for anisotropic surfaces: I.
  curved and faceted surfaces, Acta Metall. 22~(10) (1974) 1205--1214.

\bibitem{wheeler1996xi}
A.~Wheeler, G.~McFadden, A $\mat{\xi}$-vector formulation of anisotropic
  phase-field models: {3D} asymptotics, Eur. J. Appl. Math. 7~(4) (1996)
  367--381.

\bibitem{wheeler1999cahn}
A.~Wheeler, {Cahn-Hoffman} $\mat {\xi}$-vector and its relation to diffuse
  interface models of phase transitions, J. Stat. Phys. 95~(5) (1999)
  1245--1280.

\bibitem{Eggleston01}
J.~Eggleston, G.~McFadden, P.~Voorhees, A phase-field model for highly
  anisotropic interfacial energy, Physica D 150 (2001) 91--103.

\bibitem{winklmann2006note}
S.~Winklmann, A note on the stability of the {W}ulff shape, Arch. Math. 87
  (2006) 272--279.

\bibitem{Deckelnick05}
K.~Deckelnick, G.~Dziuk, C.~M. Elliott, Computation of geometric partial
  differential equations and mean curvature flow, Acta Numer. 14 (2005)
  139--232.

\bibitem{Mullins57}
W.~W. Mullins, Theory of thermal grooving, J. Appl. Phys. 28~(3) (1957)
  333--339.

\bibitem{Zhao17}
Q.~Zhao, W.~Jiang, D.~J. Srolovitz, W.~Bao, Triple junction drag effects during
  topological changes in the evolution of polycrystalline microstructures, Acta
  Mater. 128 (2017) 345--350.

\bibitem{Carter95}
W.~C. Carter, A.~R. Roosen, J.~W. Cahn, J.~E. Taylor, Shape evolution by
  surface diffusion and surface attachment limited kinetics on completely
  faceted surfaces, Acta Metall. Mater. 43~(12) (1995) 4309--4323.

\bibitem{garcke2023diffuse}
H.~Garcke, P.~Knopf, R.~N{\"u}rnberg, Q.~Zhao, A diffuse-interface approach for
  solid-state dewetting with anisotropic surface energies, J. Nonlinear Sci.
  33~(2) (2023) 34.

\bibitem{li2021energy}
Y.~Li, W.~Bao, An energy-stable parametric finite element method for
  anisotropic surface diffusion, J. Comput. Phys. 446 (2021) 110658.

\bibitem{li2023symmetrized}
M.~Li, Y.~Li, L.~Pei, A symmetrized parametric finite element method for
  simulating solid-state dewetting problems, Appl. Math. Model. 121 (2023)
  731--750.

\bibitem{bao2024structure1}
W.~Bao, Y.~Li, Q.~Zhao, A structure-preserving parametric finite element method
  for solid-state dewetting on curved substrates, arXiv preprint
  arXiv:2410.00438 (2024).

\bibitem{li2024energy}
M.~Li, C.~Zhou, Energy-stable parametric finite element approximations for
  regularized solid-state dewetting in strongly anisotropic materials, arXiv
  preprint arXiv:2407.04524 (2024).

\bibitem{duan2025solid}
Z.~Duan, M.~Li, C.~Zhou, Solid-state dewetting of axisymmetric thin film on
  axisymmetric curved-surface substrates: modeling and simulation, arXiv
  preprint arXiv:2501.00783 (2025).

\bibitem{Barrett07}
J.~W. Barrett, H.~Garcke, R.~N{\"u}rnberg, A parametric finite element method
  for fourth order geometric evolution equations, J. Comput. Phys. 222~(1)
  (2007) 441--467.

\bibitem{Barrett07b}
J.~W. Barrett, H.~Garcke, R.~N{\"u}rnberg, On the variational approximation of
  combined second and fourth order geometric evolution equations, SIAM J. Sci.
  Comput. 29~(3) (2007) 1006--1041.

\bibitem{Barrett07Ani}
J.~W. Barrett, H.~Garcke, R.~N{\"u}rnberg, Numerical approximation of
  anisotropic geometric evolution equations in the plane, SIMA J. Numer. Anal.
  28~(2) (2007) 292--330.

\bibitem{Barrett10cluster}
J.~W. Barrett, H.~Garcke, R.~N{\"u}rnberg, Parametric approximation of surface
  clusters driven by isotropic and anisotropic surface energies, Interface Free
  Bound. 12~(2) (2010) 187--234.

\bibitem{Barrett10}
J.~W. Barrett, H.~Garcke, R.~N{\"u}rnberg, Finite-element approximation of
  coupled surface and grain boundary motion with applications to thermal
  grooving and sintering, Eur. J. Appl. Math. 21~(6) (2010) 519--556.

\bibitem{bao2023structure}
W.~Bao, H.~Garcke, R.~N{\"u}rnberg, Q.~Zhao, A structure-preserving finite
  element approximation of surface diffusion for curve networks and surface
  clusters, Numerical Methods for Partial Differential Equations 39~(1) (2023)
  759--794.

\bibitem{Sutton95}
A.~P. Sutton, R.~W. Balluffi, Interfaces in crystalline materials, Clarendon
  Press, 1995.

\bibitem{bao2023symmetrized}
W.~Bao, W.~Jiang, Y.~Li, A symmetrized parametric finite element method for
  anisotropic surface diffusion of closed curves, SIAM J. Numer. Anal. 61~(2)
  (2023) 617--641.

\bibitem{li2025structure}
Y.~Li, W.~Ying, Y.~Zhang, A structure-preserving parametric finite element
  method with optimal energy stability condition for anisotropic surface
  diffusion, arXiv preprint arXiv:2501.16660 (2025).

\bibitem{jiang2021perimeter}
W.~Jiang, B.~Li, A perimeter-decreasing and area-conserving algorithm for
  surface diffusion flow of curves, J. Comput. Phys. 443 (2021) 110531.

\bibitem{bao2021structure}
W.~Bao, Q.~Zhao, A structure-preserving parametric finite element method for
  surface diffusion, SIAM J. Numer. Anal. 59~(5) (2021) 2775--2799.

\end{thebibliography}
\end{document}